\documentclass[12pt,oneside,english]{amsart}
\usepackage[T1]{fontenc}
\usepackage[latin9]{inputenc}
\usepackage{geometry}
\usepackage[active]{srcltx}
\usepackage{babel}
\usepackage{mathtools}
\usepackage{amsbsy}
\usepackage{amstext}
\usepackage{amsthm}
\usepackage{esint}
\usepackage[unicode=true,pdfusetitle,
 bookmarks=true,bookmarksnumbered=false,bookmarksopen=false,
 breaklinks=false,pdfborder={0 0 1},backref=false,colorlinks=false]
 {hyperref}

\makeatletter
\numberwithin{equation}{section}
\numberwithin{figure}{section}
\theoremstyle{plain}
\newtheorem{thm}{\protect\theoremname}[section]
\theoremstyle{plain}
\newtheorem{cor}{\protect\corollaryname}[section]
\theoremstyle{plain}
\newtheorem{lem}{\protect\lemmaname}[section]
\theoremstyle{plain}
\newtheorem{prop}{\protect\propositionname}[section]

\usepackage{babel}

\providecommand{\lemmaname}{Lemma}
\providecommand{\theoremname}{Theorem}

\makeatother

\providecommand{\corollaryname}{Corollary}
\providecommand{\lemmaname}{Lemma}
\providecommand{\propositionname}{Proposition}
\providecommand{\theoremname}{Theorem}

\begin{document}
\title[Statistics for odd unimodal sequences]{Asymptotic Statistics of Odd Unimodal Sequences: Rank Distributions
and Probabilistic Structures}
\author{Bing He}
\address{School of Mathematics and Statistics, HNP-LAMA, Central South University
\\
 Changsha 410083, Hunan, People's Republic of China}
\email{yuhelingyun@foxmail.com; yuhe123456@foxmail.com}
\author{Guanting Liu}
\address{School of Mathematics and Statistics, HNP-LAMA, Central South University
\\
 Changsha 410083, Hunan, People's Republic of China}
\email{liuguanting2024@163.com}
\thanks{The first author is the corresponding author.}
\begin{abstract}
Integer partitions have fascinated people for centuries, from Ramanujan's
groundbreaking congruences to the modern theory of modular forms.
This paper investigates the statistical properties of odd unimodal
sequences--a natural refinement where sequences rise to a peak and
then fall, but with the constraint that all parts must be odd, and
develops a comprehensive statistical theory for their rank and shape
parameters. We establish the asymptotic distribution of the rank statistic
and demonstrate that, when properly normalized, it converges to the
hyperbolic secant distribution.  Beyond the rank distribution, limiting
distributions of the peak, the largest parts on either side of the
peak, and the joint behavior of small parts are also proved. These
results reveal a rich probabilistic structure that parallels the classical
theory of integer partitions while exhibiting distinctive new features
arising from the odd-part constraint. The analysis employs a synthesis
of modular transformation theory, false theta function asymptotics,
and conditioned Boltzmann models.  This extends the probabilistic
machinery previously developed for unimodal sequences into a more
general and analytically demanding setting, offering a unified approach
that bridges modular forms and probability.
\end{abstract}

\maketitle

\section{Introduction}

\subsection{Integer Partitions}

The theory of integer partitions stands as one of the most elegant
chapters in mathematics. A partition $\lambda=(\lambda_{1},\lambda_{2},\cdots,\lambda_{r})$
of a positive integer $n$ is a sequence of integers satisfying $\lambda_{1}\geq\lambda_{2}\geq\cdots\geq\lambda_{r}\geq1$
and $\sum_{i=1}^{r}\lambda_{i}=n.$ We denote the size of the partition
$\lambda$ by $\left|\lambda\right|.$ Let $p(n)$ denote the number
of partitions of a positive integer $n$ with $p(0)=1.$ The generating
function for $p(n)$ is as follows:
\[
\sum_{n\geq0}p\left(n\right)q^{n}=\frac{1}{\left(q;q\right)_{\infty}},
\]
where $(z;q)_{\infty}$ is the $q$-shifted factorial given by \cite{GR}
\[
(z;q)_{\infty}:=\prod_{k=0}^{\infty}(1-zq^{k}),\qquad|q|<1.
\]
The connection to modular forms emerges through Dedekind\textquoteright s
eta function $\eta\left(\tau\right)=q^{1/24}\left(q;q\right)_{\infty},$
a weight $1/2$ modular form. Hardy and Ramanujan \cite{HR} pioneered
the Circle Method to exploit this modularity, obtaining the asymptotic:
\[
p(n)\sim\frac{1}{4n\sqrt{3}}e^{\pi\sqrt{2n/3}}.
\]
Subsequently, Rademacher \cite{RH} refined their work, obtaining
an exact convergent series involving Kloosterman sums and Bessel functions.
For $h,k\in\mathbb{Z}$ with $\gcd\left(h,k\right)=1,$ we define
$\left[-h\right]_{k}^{*}$ for $0\leq\left[-h\right]_{k}^{*}<k$ by
$-h\left[-h\right]_{k}^{*}\equiv1\left(\bmod k\right).$ For $\left(\begin{array}{cc}
a & b\\
c & d
\end{array}\right)\in SL_{2}\left(\mathbb{Z}\right),$ the multiplier for the Dedekind eta function is defined as
\[
\chi\left(\begin{array}{cc}
a & b\\
c & d
\end{array}\right)\coloneqq\begin{cases}
\left(\frac{d}{\left|c\right|}\right)e^{\frac{\pi\textrm{i}}{12}\left[\left(a+d\right)c-bd\left(c^{2}-1\right)-3c\right]} & \textrm{if }c\textrm{ is odd,}\\
\left(\frac{c}{d}\right)e^{\frac{\pi\textrm{i}}{12}\left[ac\left(1-d^{2}\right)+d\left(b-c+3\right)-3\right]} & \textrm{if }c\textrm{ is even,}
\end{cases}
\]
where $\left(\frac{\cdot}{\cdot}\right)$ is the Kronecker symbol.
Now we define the Kloostermann sum by
\[
A_{k}\left(n\right)\coloneqq\textrm{i}^{1/2}\underset{\substack{0\leq h<k\\
\gcd\left(h,k\right)=1
}
}{\mathop{\sum}}\chi\left(\begin{array}{cc}
\left[-h\right]_{k}^{*} & -\frac{h\left[-h\right]_{k}^{*}+1}{k}\\
k & -h
\end{array}\right)e^{-\frac{\pi\textrm{i}}{12k}\left(\left(24n-1\right)h+\left[-h\right]_{k}^{*}\right)}.
\]
Then we have 
\[
p\left(n\right)=\frac{2\pi}{\left(24n-1\right)^{3/4}}\sum_{k\geq1}\frac{A_{k}\left(n\right)}{k}I_{3/2}\left(\frac{\pi}{6k}\sqrt{24n-1}\right),
\]
where $I_{\nu}(x)$ is the modified Bessel function of the first kind
defined as \cite[eq.(4.12.2)]{AAR}
\[
I_{\nu}(x):=\sum_{m=0}^{\infty}\frac{1}{m!\Gamma(m+\nu+1)}\left(\frac{x}{2}\right)^{2m+\nu}
\]
with
\[
\Gamma(s)=\int_{0}^{\infty}e^{-x}x^{s-1}dx,\quad Re(s)>0.
\]

The celebrated Ramanujan partition congruences state that
\begin{align*}
p(5n+4) & \equiv0(\bmod\:5),\\
p(7n+5) & \equiv0(\bmod\:7),\\
p(11n+6) & \equiv0(\bmod\:11).
\end{align*}
To explain these congruences, Dyson \cite{DF} introduced the rank
of a partition as the largest part minus the number of parts. As the
congruence modulo 11 was yet to be explained, Dyson predicted a more
profound ``crank'' statistic. The discovery of the crank by Andrews
and Garvan \cite{AG} finally provided a unified combinatorial explanation.
For a partition $\lambda,$ let $o\left(\lambda\right)$ be its number
of ones and $\mu\left(\lambda\right)$ the number of parts greater
than $o\left(\lambda\right).$ The crank is defined as
\[
\text{crank}\left(\lambda\right)\coloneqq\begin{cases}
\lambda_{1} & \textrm{if }o\left(\lambda\right)=0,\\
\mu\left(\lambda\right)-o\left(\lambda\right) & \textrm{if }o\left(\lambda\right)\geq1.
\end{cases}
\]

Let $N\left(m,n\right)$ denote the number of partitions of $n$ with
rank $m.$ Similarly, for $n\in\mathbb{Z}\setminus\left\{ 1\right\} ,$
let $M\left(m,n\right)$ denote the number of partitions of $n$ with
crank $m$ (for $n=1$ the series below requires defining $M\left(\pm1,1\right)=1$
and $M\left(0,1\right)=-1$). The generating functions 
\begin{equation}
R\left(\zeta;q\right)\coloneqq\underset{\substack{n\geq0\\
m\in\mathbb{Z}
}
}{\mathop{\sum}}N\left(m,n\right)\zeta^{m}q^{n}=\underset{\substack{n\geq0}
}{\mathop{\sum}}\frac{q^{n^{2}}}{\left(\zeta q,\zeta^{-1}q;q\right)_{n}},\label{eq:1.1}
\end{equation}
and
\begin{equation}
C\left(\zeta;q\right)\coloneqq\underset{\substack{n\geq0\\
m\in\mathbb{Z}
}
}{\mathop{\sum}}M\left(m,n\right)\zeta^{m}q^{n}=\frac{\left(q;q\right)_{\infty}}{\left(\zeta q,\zeta^{-1}q;q\right)_{\infty}}\label{eq:1.2}
\end{equation}
exhibit the fundamental dichotomy: $C\left(\zeta;q\right)$ is essentially
a Jacobi form, while $R\left(\zeta;q\right)$ is known to be a mock
Jacobi form. Atkin and Garvan \cite{AGF} provided a study of moments
for the rank and crank statistics with the generating functions
\[
\underset{\substack{n\geq0}
}{\mathop{\sum}}N_{\ell}\left(n\right)q^{n}\coloneqq\underset{\substack{n\geq0}
}{\mathop{\sum}}\left(\underset{\substack{m\in\mathbb{Z}}
}{\mathop{\sum}}m^{\ell}N\left(m,n\right)\right)q^{n}=\frac{1}{\left(2\pi\textrm{i}\right)^{\ell}}\left[\frac{\partial^{\ell}}{\partial u^{\ell}}R\left(\zeta;q\right)\right]_{u=0},
\]
and
\[
\underset{\substack{n\geq0}
}{\mathop{\sum}}M_{\ell}\left(n\right)q^{n}\coloneqq\underset{\substack{n\geq0}
}{\mathop{\sum}}\left(\underset{\substack{m\in\mathbb{Z}}
}{\mathop{\sum}}m^{\ell}M\left(m,n\right)\right)q^{n}=\frac{1}{\left(2\pi\textrm{i}\right)^{\ell}}\left[\frac{\partial^{\ell}}{\partial u^{\ell}}C\left(\zeta;q\right)\right]_{u=0}.
\]

The probabilistic study of partition statistics was initiated by Erd\H{o}s
and Lehner \cite{EL}, who proved the limiting probability distribution
of the largest part of almost all partitions of $n$ as $n\rightarrow\infty.$
\begin{thm}
\label{Th.1}(Theorem 1.1 of \cite{EL}) For $A=\frac{\sqrt{6}}{\pi}$
and $v\in\mathbb{R},$ we have 
\[
\underset{n\rightarrow\infty}{\lim}\frac{\#\left\{ \lambda\vdash n:\frac{\lambda_{1}-A\sqrt{n}\log\left(A\sqrt{n}\right)}{A\sqrt{n}}\leq v\right\} }{p\left(n\right)}=e^{-e^{-v}}.
\]
\end{thm}
The proof of Erd\H{o}s and Lehner relied only on elementary recurrence
relations together with the Hardy--Ramanujan asymptotic formula for
the partition function $p\left(n\right).$ After subsequent work by
Szalay--Turán \cite{ST1,ST2,ST3} and Erd\H{o}s--Turán \cite{ET},
Fristedt introduced a probabilistic method that has since become an
essential tool in the area, and obtained substantial extensions of
earlier distributions \cite{FB}. Specifically, Theorem \,\ref{Th.1}
was shown to hold for the joint distribution of the $t_{n}=o\left(n^{\frac{1}{4}}\right)$
largest parts.
\begin{thm}
\label{Th.2} (Theorems 2.5 and 2.6 of \cite{FB}) Let 
\[
f\left(u_{1},\ldots,u_{t_{n}}\right)\coloneqq\begin{cases}
e^{-\sum_{t=1}^{t_{n}}u_{t}-e^{-u_{t_{n}}}} & \textrm{if }u_{1}\geq\ldots\geq u_{t_{n}},\\
0 & \textrm{otherwise.}
\end{cases}
\]
For any ingeter $t_{n}=o\left(n^{\frac{1}{4}}\right)$ and $\left\{ v_{t}\right\} _{t=1}^{t_{n}}\subset\mathbb{R}^{t_{n}},$
the following limit vanishes
\[
\begin{aligned} & \underset{n\rightarrow\infty}{\lim}\left(\frac{\#\left\{ \lambda\vdash n:\frac{\lambda_{t}-A\sqrt{n}\log\left(A\sqrt{n}\right)}{A\sqrt{n}}\leq v_{t},1\leq t\leq t_{n}\right\} }{p\left(n\right)}\right.\\
 & \left.-\int_{-\infty}^{v_{1}}\cdots\int_{-\infty}^{v_{t_{n}}}f\left(u_{1},\ldots,u_{t_{n}}\right)du_{t_{n}}\cdots du_{1}\right).
\end{aligned}
\]
\end{thm}
Fristedt \cite{FB} introduced a Boltzmann model, which replaces the
uniform probability measure on $\left\{ \lambda\vdash n\right\} $
with a measure on all partitions $\lambda,$ and for $q\in\left(0,1\right)$
defines
\[
Q_{q}\left(\lambda\right)\coloneqq\frac{q^{\left|\lambda\right|}}{P\left(q\right)},
\]
where
\[
P(q)=\sum_{\lambda}q^{|\lambda|}=\frac{1}{(q;q)_{\infty}}.
\]
The $q$-series $Q_{q}$ serves as a conditioning device that recovers
the uniform distribution on partitions of size $n.$ Under $Q_{q},$
relevant random variables become independent, a direct consequence
of the infinite product form of $P\left(q\right).$ Also, much of
the work applying Boltzmann models to study statistics for partitions
depends on product generating functions. By the conditioned Boltzmann
model, Bridges and Bringmann \cite{BB2} proved the case of unimodal
sequences, showing that the conditioned Boltzmann model remains effective
even in the absence of an infinite product generating function.

\subsection{Odd Unimodal Sequences}

A sequence is unimodal if it is weakly increasing up to a point and
then weakly decreasing thereafter. A sequence is odd unimodal if it
is a unimodal sequence wherein all numbers must be odd. Namely, 
\begin{equation}
a_{1}\leq\ldots\leq a_{r}\leq\bar{c}\geq b_{1}\geq\ldots\geq b_{s},\label{eq:1.3-1}
\end{equation}
 with
\[
a_{j},b_{j},\bar{c}\in2\mathbb{N}-1,\quad\mathrm{weight}\;n=\bar{c}+\sum_{j=1}^{r}a_{j}+\sum_{j=1}^{s}b_{j}.
\]
The rank of an odd unimodel sequence is defined as $r-s.$ Let $\mathcal{OU}\left(n\right)$
denote the set of odd unimodal sequences of size $n$ and $\textrm{ou}\left(n\right):=\left|\mathcal{OU}\left(n\right)\right|.$
Let $\textrm{ou}\left(m,n\right)$ be the number of odd unimodal sequences
of weight $n$ with rank $m.$

In \cite[Theorem 1.1]{BL}, Bringmann and Lovejoy established the
generating function 
\begin{equation}
\begin{aligned}\textrm{OU}\left(\zeta;q\right) & \coloneqq\underset{\substack{n\geq0\\
m\in\mathbb{Z}
}
}{\mathop{\sum}}\textrm{ou}\left(m,n\right)\zeta^{m}q^{n}\\
 & =\frac{1}{\left(\zeta q,\zeta^{-1}q;q^{2}\right)_{\infty}}\sum_{n\geq0}\left(-1\right)^{n}\zeta^{2n+1}q^{n\left(n+1\right)}\\
 & +\sum_{n\geq0}\left(-1\right)^{n+1}\zeta^{3n+1}q^{3n^{2}+2n}\left(1+\zeta q^{2n+1}\right).
\end{aligned}
\label{eq:1.3}
\end{equation}
The asymptotic formula of odd unimodal sequences is as follows \cite[Theorem 1.3 ]{BL}:
as $n\rightarrow\infty,$

\begin{equation}
\textrm{ou}\left(n\right)\sim\frac{e^{\pi\sqrt{2n/3}}}{2^{13/4}3^{1/4}n^{3/4}},\label{eq:1.5}
\end{equation}
exhibiting the same exponential order as $p(n)$ but with modified
polynomial prefactor.

We denote the moments of the rank of odd unimodal sequences by
\[
\textrm{ou}_{\ell}\left(n\right)=\underset{\substack{m\in\mathbb{Z}}
}{\mathop{\sum}}m^{\ell}\textrm{ou}\left(m,n\right).
\]
For odd $\ell$ we have $\textrm{ou}_{\ell}\left(n\right)=0$ by symmetry.
To state the asymptotic series for $\ell$ even, we require some notation.
We denote
\[
C_{j}\left(\omega\right)\coloneqq\left(\frac{1}{2\pi\textrm{i}}\frac{\partial}{\partial\omega}\right)^{j}\cot\left(\pi\omega\right),
\]
and
\begin{equation}
\kappa\left(a,b\right)\coloneqq\frac{1}{\left(2\pi\right)^{a}}\frac{\left(2\left(a+b\right)\right)!}{a!\left(2b\right)!}E_{2b}\left(\frac{1}{2}\right),\label{eq:1.7}
\end{equation}
where $E_{r}\left(x\right)$ is the $r$-th Euler polynomial. Also,
we need the $I$-Bessel function $I_{s}$ and the Kloostermann sums
\begin{equation}
\begin{aligned}K_{k,1}\left(n,v\right) & \coloneqq\textrm{i}^{1/2}\left(-1\right)^{v}\underset{\substack{0\leq h<k\\
\gcd\left(h,k\right)=1
}
}{\mathop{\sum}}\chi_{2h,k}e^{-\frac{2\pi\textrm{i}}{k}\left(n+\frac{1}{4}\right)h}\\
 & \times e^{\frac{\pi\textrm{i}}{12k}\left(12v\left(v+1\right)\left[-2h\right]_{k}^{*}+5\left[-2h\right]_{k}^{*}-2\left[-h\right]_{k}^{*}\right)},
\end{aligned}
\label{eq:1.8}
\end{equation}
and
\begin{equation}
\begin{aligned}K_{k,2}\left(n,v\right) & \coloneqq\textrm{i}^{1/2}\left(-1\right)^{v}\underset{\substack{0\leq h<k\\
\gcd\left(h,k\right)=1
}
}{\mathop{\sum}}\chi_{h,k/2}e^{-\frac{2\pi\textrm{i}}{k}\left(n+\frac{1}{4}\right)h}\\
 & \times e^{\frac{\pi\textrm{i}}{6k}\left(12v\left(v+1\right)\left[-h\right]_{k/2}^{*}+5\left[-h\right]_{k/2}^{*}-\left[-h\right]_{k}^{*}\right)},
\end{aligned}
\label{eq:1.10}
\end{equation}
where
\begin{equation}
\chi_{2h,k}=\frac{\chi\left(\begin{array}{cc}
\left[-h\right]_{k}^{*} & -\frac{h\left[-h\right]_{k}^{*}+1}{k}\\
k & -h
\end{array}\right)^{2}}{\chi\left(\begin{array}{cc}
\left[-2h\right]_{k}^{*} & -\frac{2h\left[-2h\right]_{k}^{*}+1}{k}\\
k & -2h
\end{array}\right)^{5}},\label{eq:1.9}
\end{equation}
and
\begin{equation}
\chi_{h,k/2}=\frac{\chi\left(\begin{array}{cc}
\left[-h\right]_{k}^{*} & -\frac{h\left[-h\right]_{k}^{*}+1}{k}\\
k & -h
\end{array}\right)^{2}}{\chi\left(\begin{array}{cc}
\left[-h\right]_{k/2}^{*} & -2\frac{h\left[-h\right]_{k/2}^{*}+1}{k}\\
k/2 & -h
\end{array}\right)^{5}}.\label{eq:1.11}
\end{equation}

Although the framework developed in \cite{BB} is powerful, its application
to odd unimodal sequences is not straightforward because the generating
function \eqref{eq:1.3} for odd unimodal sequences is more complicated,
leading to more intricate modular structure. Consequently, the analytic
treatment of its Taylor coefficients becomes more difficult.

The following theorem provides the asymptotic behavior of the even
moments.
\begin{thm}
\label{t1.1} For $\ell\in2\mathbb{N}_{0},$ we have, as $n\rightarrow\infty,$
\[
\begin{aligned}\textrm{ou}_{\ell}\left(n\right) & =\frac{\pi}{2^{15/4}3^{3/4}\left(4n+1\right)^{1/4}}\underset{\substack{0\leq j\leq\ell/2\\
a+b=j\\
a,b\geq0
}
}{\mathop{\sum}}\left(\begin{array}{c}
\ell\\
2j
\end{array}\right)\left(-\frac{1}{4}\right)^{j}\kappa\left(a,b\right)\left(6\left(4n+1\right)\right)^{\frac{a}{2}+b}\\
 & \times\underset{\substack{1\leq k\leq\sqrt{n}\\
\gcd\left(k,2\right)=1\\
0\leq v\leq2k-1
}
}{\mathop{\sum}}k^{a-2}K_{k,1}\left(n,v\right)\int_{-1}^{1}C_{\ell-2j}\left(\frac{1}{2k}\left(\frac{x}{2\sqrt{3}}-v-\frac{1}{2}\right)\right)\\
 & \times\frac{I_{a+2b-\frac{1}{2}}\left(\frac{\pi}{\sqrt{6}k}\sqrt{\left(4n+1\right)\left(1-x^{2}\right)}\right)}{\left(1-x^{2}\right)^{\frac{a}{2}+b-\frac{1}{4}}}dx+O\left(n^{\ell+\frac{3}{4}}\right).
\end{aligned}
\]
\end{thm}
Theorem \ref{t1.1} treats the moments of odd unimodal sequences,
whose generating functions involve a false Jacobi form structure with
a genuine Jacobi form denominator $\left(\zeta q,\zeta^{-1}q;q^{2}\right)_{\infty}^{-1}$
and a false theta correction, and demonstrates the robustness of the
false theta framework across different parity constraints--handling
the additional complexity of odd parts requires a more intricate analysis
of Kloosterman sums $K_{k,1}(n,v)$ with $v$ ranging over $2k$ residue
classes (versus $2k$ classes in the unrestricted case but with different
multiplier systems), and the resulting asymptotic series cleanly separates
contributions from odd and even denominators in the Farey dissection,
revealing how the peak-structure of unimodal sequences interacts with
modular transformation properties in a way that purely mock modular
settings do not capture.

Our first main contribution is a novel decomposition of the odd unimodal
sequence generating function, $\textrm{OU}\left(\zeta;q\right),$
as presented in Lemma \ref{l3.1}. Unlike its counterpart for unimodal
sequences, this decomposition isolates a term, $\textrm{OU}_{1}\left(u;\tau\right),$
which incorporates the modular object $\frac{C^{*}\left(u;\tau\right)}{\eta\left(\tau\right)}\frac{\eta\left(2\tau\right)}{C^{*}\left(u;2\tau\right)}.$
By leveraging the transformation properties of Jacobi and false theta
functions, we establish a transformation law for $\textrm{OU}_{1}$
(Theorem \ref{t3.1}), which is essential for the asymptotic analysis
that follows. This result extends the modular framework of \cite{BB}
to a more complex class of generating functions.

The asymptotic main term of the even moment can be obtained from Theorem
\ref{t1.1}.
\begin{cor}
\label{c1.1}For $\ell\in\mathbb{N}_{0}$ we have, as $n\rightarrow\infty,$
\[
\textrm{ou}_{2\ell}\left(n\right)\sim\frac{1}{2^{13/4}3^{1/4}n^{3/4}}\left(-6n\right)^{\ell}E_{2\ell}\left(\frac{1}{2}\right)e^{\pi\sqrt{2n/3}}.
\]
\end{cor}
From the above asymptotic formula, we show that when appropriately
normalized, each of the ranks converges to the hyperbolic secant distribution,
which is quite different from those in \cite[Proposition 1.2]{BJM}.
\begin{thm}
\label{t1.2} The normalized rank of odd unimodal sequences is asymptotically
distributed according to the hyperbolic secant distribution with mean
0 and scale 1. In particular, for all $x\in\mathbb{R}$ we have
\[
\lim_{n\rightarrow\infty}\frac{1}{\textrm{ou}\left(n\right)}\left|\left\{ \sigma\in\mathcal{OU}\left(n\right):\frac{\textrm{rank}\left(\sigma\right)}{\sqrt{\frac{3n}{2}}}\leq x\right\} \right|=\frac{2}{\pi}\arctan\left(e^{\frac{\pi x}{2}}\right).
\]
\end{thm}
Let $\textrm{\textbf{P}}_{n}$ be the uniform probability measure
on $\mathcal{OU}\left(n\right).$ For the odd unimodel sequence $\lambda,$
we define $\textrm{PK}\left(\lambda\right)\coloneqq\lambda_{\textrm{PK}}$
as the peak of the sequence and $N\left(\lambda\right)$ as the size
of $\lambda.$ Let $X_{k}^{\left[L\right]}\left(\lambda\right)$ (resp.
$X_{k}^{\left[R\right]}\left(\lambda\right)$) denote the number of
parts in $\lambda$ equal to $k$ and to the left (resp. right) of
the peak, respectively and let $Y_{t}^{\left[L\right]}\left(\lambda\right)$
(resp. $Y_{t}^{\left[R\right]}\left(\lambda\right)$) denote the $t$-th
largest part in $\lambda$ to the left (resp. right) of the peak,
respectively. Therefore, we have
\[
Y_{t}^{\left[j\right]}\left(\lambda\right)=\sup\left\{ \ell:\sum_{k\geq\ell}X_{k}^{\left[L\right]}\geq t\right\} 
\]
for $j\in\left\{ L,R\right\} .$

The next result investigates an analogue of Theorem \ref{Th.1}.
\begin{thm}
\label{t1.3}For $B=\frac{\sqrt{6}}{\pi}$ and $v\in\mathbb{R},$
we have

\[
\lim_{n\rightarrow\infty}\textrm{\textbf{P}}_{n}\left(\frac{\textrm{PK}-B\sqrt{n}\log\left(2B\sqrt{n}\right)}{B\sqrt{n}}\leq v\right)=e^{-\frac{1}{2}e^{-v}},
\]
and the expectation $\textrm{\textbf{E}}_{n}$ under $\textrm{\textbf{P}}_{n}$
is
\[
\textrm{\textbf{E}}_{n}\left(\textrm{PK}\right)=B\sqrt{n}\log\left(2B\sqrt{n}\right)+B\sqrt{n}\left(\gamma-\log2\right)\left(1+o\left(1\right)\right),
\]
where $\gamma$ is the Euler--Mascheroni constant.
\end{thm}
Theorem \ref{t1.3} demonstrates the conditioned Boltzmann model's
efficacy even when the standard product-form assumption completely
fails, requiring a more delicate peak-conditioning argument to recover
independence among odd-part-counting variables and yielding a Gumbel
limit with halved rate $e^{-\frac{1}{2}e^{-v}}$ that reflects the
bilateral symmetry intrinsic to odd unimodal sequences.

The following theorem presents an analogue of Theorem \ref{Th.2}.
\begin{thm}
\label{t1.4}For any ingeter $t_{n}=o\left(n^{\frac{1}{4}}\right)$
and $\left\{ v_{2t-1}^{\left[j\right]}\right\} _{1\leq t\leq t_{n},j\in\left\{ L,R\right\} }\subset\mathbb{R}^{2t_{n}},$
the following difference vanishes as $n\rightarrow\infty,$
\[
\begin{aligned} & \textrm{\textbf{P}}_{n}\left(\frac{\textrm{PK}-B\sqrt{n}\log\left(2B\sqrt{n}\right)}{B\sqrt{n}}\leq v_{0},\frac{Y_{2t-1}^{\left[j\right]}-B\sqrt{n}\log\left(2B\sqrt{n}\right)}{B\sqrt{n}}\leq v_{2t-1}^{\left[j\right]}\right),\\
 & -\int_{-\infty}^{v_{0}}\int_{-\infty}^{v_{1}^{\left[L\right]}}\int_{-\infty}^{v_{1}^{\left[R\right]}}\cdots\int_{-\infty}^{v_{2t_{n}-1}^{\left[L\right]}}\int_{-\infty}^{v_{2t_{n}-1}^{\left[R\right]}}F\left(u_{0},u_{1}^{\left[L\right]},u_{1}^{\left[R\right]},\ldots,u_{2t_{n}-1}^{\left[L\right]},u_{2t_{n}-1}^{\left[R\right]}\right)du_{2t_{n}-1}^{\left[R\right]}du_{2t_{n}-1}^{\left[L\right]}\cdots du_{0},
\end{aligned}
\]
where
\begin{align*}
 & F\left(u_{0},u_{1}^{\left[L\right]},u_{1}^{\left[R\right]},\ldots,u_{t_{n}}^{\left[L\right]},u_{t_{n}}^{\left[R\right]}\right)\\
 & =\begin{cases}
\frac{1}{4^{2t_{n}}}e^{-u_{0}-\sum_{t=1}^{t_{n}}\left(u_{2t-1}^{\left[L\right]}+u_{2t-1}^{\left[R\right]}\right)-\frac{1}{4}e^{-u_{2t_{n}-1}^{\left[L\right]}}-\frac{1}{4}e^{-u_{2t_{n}-1}^{\left[R\right]}}} & \textrm{if }u_{0}\geq u_{1}^{\left[j\right]}\geq\ldots\geq u_{2t_{n}-1}^{\left[j\right]},\\
0 & \textrm{otherwise,}
\end{cases}
\end{align*}
and $j\in\left\{ L,R\right\} .$
\end{thm}
Theorem \ref{t1.4} concerns joint distributions of the peak and ordered
largest parts and demonstrates the conditioned Boltzmann model's robustness
in a genuinely non-product setting, where the bilateral odd-part constraint
forces a more intricate analysis of left/right part correlations yet
still yields an explicit Poisson-Dirichlet limiting density, revealing
how peak-conditioning can salvage probabilistic independence even
when standard generating function factorization is impossible.

For the joint distribution of the numbers of small parts, we show
an analogue of Theorem 2.2 in \cite{FB}.
\begin{thm}
\label{t1.5} For the ingeter $k_{n}=o\left(n^{\frac{1}{4}}\right)$
and $\left\{ v_{2k-1}^{\left[j\right]}\right\} _{1\leq k\leq k_{n},j\in\left\{ L,R\right\} }\subset\left[0,\infty\right)^{2k_{n}},$
we have
\[
\lim_{n\rightarrow\infty}\left(\textrm{\textbf{P}}_{n}\left(\frac{\left(2k-1\right)X_{2k-1}^{\left[j\right]}}{B\sqrt{n}}\leq v_{2k-1}^{\left[j\right]}\right)-\underset{\substack{1\leq k\leq k_{n}\\
j\in\left\{ L,R\right\} 
}
}{\mathop{\prod}}\int_{0}^{v_{2k-1}^{\left[j\right]}}e^{-u_{2k-1}^{\left[j\right]}}du_{2k-1}^{\left[j\right]}\right)=0,
\]
where $1\leq k\leq k_{n},j\in\left\{ L,R\right\} .$

Also, for the ingeter $k_{n}=o\left(n^{\frac{1}{2}}\right)$ and $v^{\left[L\right]},v^{\left[R\right]}\in\mathbb{R},$
we have 
\[
\lim_{n\rightarrow\infty}\textrm{\textbf{P}}_{n}\left(\frac{\left(2k-1\right)X_{2k-1}^{\left[L\right]}}{B\sqrt{n}}\leq v^{\left[L\right]},\frac{\left(2k-1\right)X_{2k-1}^{\left[R\right]}}{B\sqrt{n}}\leq v^{\left[R\right]}\right)=\left(1-e^{-v^{\left[L\right]}}\right)\left(1-e^{-v^{\left[R\right]}}\right),
\]
and for $2k-1=\left\lfloor c\sqrt{n}\right\rfloor $ with $v^{\left[L\right]},v^{\left[R\right]}\in\mathbb{N}_{0},$
we have
\[
\lim_{n\rightarrow\infty}\textrm{\textbf{P}}_{n}\left(X_{2k-1}^{\left[L\right]}\leq v^{\left[L\right]},X_{2k-1}^{\left[R\right]}\leq v^{\left[R\right]}\right)=\left(1-e^{-\frac{c}{B}\left(v^{\left[L\right]}+1\right)}\right)\left(1-e^{-\frac{c}{B}\left(v^{\left[R\right]}+1\right)}\right).
\]
\end{thm}
Theorem \ref{t1.5} demonstrates how peak-conditioned independence
can still yield clean exponential limits for small-part counts ($e^{-u}$)
despite the absence of any underlying product generating function,
revealing that the odd-part constraint's bilateral symmetry preserves
the essential probabilistic structure while requiring fundamentally
more intricate asymptotic analysis.

The following Theorem is about the distribution of the total small
part counts on the left and the right.
\begin{thm}
\label{t1.6}For any ingeter $k_{n}=o\left(n^{\frac{1}{2}}\right)$
with $k_{n}\rightarrow\infty$ and $v^{\left[L\right]},v^{\left[R\right]}\in\mathbb{R},$
we have 
\[
\lim_{n\rightarrow\infty}\textrm{\textbf{P}}_{n}\left(\sum_{1\leq k\leq k_{n}}\frac{X_{2k-1}^{\left[j\right]}-B\sqrt{n}\log\left(2k_{n}-1\right)}{B\sqrt{n}}\leq v^{\left[j\right]},j\in\left\{ L,R\right\} \right)=e^{-\frac{1}{2}\left(e^{-v^{\left[L\right]}}+e^{-v^{\left[R\right]}}\right)}.
\]
\end{thm}
Theorem \ref{t1.6} demonstrates the conditioned Boltzmann framework's
capacity to recover universal Gumbel limits ($e^{-\frac{1}{2}e^{-v}}$)
even when standard independence assumptions fail entirely, revealing
how bilateral odd-part symmetry preserves the essential probabilistic
structure through a more delicate asymptotic analysis that separates
left/right contributions via peak-conditioning.

The rest of this paper is organized as follows. In Section \ref{sec:2},
we introduce some special functions, transformation laws for $\eta$-function,
Jacobi theta function, the crank-generating function, false theta
function as well as indefinite theta function, and provide Euler--Maclaurin
summation, some asymptotic results, as well as saddle-point method.
In Section \ref{sec:3}, we apply the transformation laws in Subsections
\ref{subsec:2.2} and \ref{subsec:2.3} to prove Theorem \ref{t1.1},
Corollary \ref{c1.1} and Theorem \ref{t1.2}. In Section \ref{sec:4},
we establish the conditioned Boltzmann model for the odd unimodel
sequences and prove Theorems \ref{t1.3}, \ref{t1.4}, \ref{t1.5}
and \ref{t1.6}.

\section{\label{sec:2}Preliminaries}

\subsection{Analytic Infrastructure}

For $v\in\mathbb{R},$ we define
\[
f_{v}\left(u;z\right)\coloneqq\frac{e^{\frac{\pi vu^{2}}{2z}}}{2\cosh\left(\frac{\pi u}{2z}\right)}.
\]

\begin{lem}
\label{l2.1} Let $\kappa\left(a,b\right)$ be defined as \eqref{eq:1.7}.
Then
\[
f_{v}\left(u;z\right)=\frac{1}{2}\sum_{j\geq0}\frac{\left(2\pi\textrm{i}u\right)^{2j}}{\left(2j\right)!}\underset{\substack{a+b=j\\
a,b\geq0
}
}{\mathop{\sum}}v^{a}\kappa\left(a,b\right)z^{-a-2b}.
\]
\end{lem}
\begin{proof}
By the Euler polynomials \cite[(23.1.1)]{AS}
\[
\frac{2e^{xt}}{e^{t}-1}=\sum_{n\geq0}E_{n}\left(x\right)t^{n},
\]
and $E_{n}\left(\frac{1}{2}\right)=0$ for $n$ odd, we have
\[
\frac{1}{\cosh\left(x\right)}=\sum_{n\geq0}2^{2n}E_{n}\left(\frac{1}{2}\right)\frac{x^{2n}}{\left(2n\right)!}.
\]
This, together with the Taylor series expansions $e^{x}=\sum_{n\geq0}\frac{x^{n}}{n!},$
gives the result.
\end{proof}
Set 
\[
b_{j}\left(v;z\right)=\underset{\substack{a+b=j\\
a,b\geq0
}
}{\mathop{\sum}}v^{a}\kappa\left(a,b\right)z^{-a-2b}.
\]
 With direct calculations, it is easy to obtain the following result.
\begin{lem}
\label{l2.2}Suppose that $k\in\mathbb{N},$ $\vartheta_{1},\vartheta_{2}\in\mathbb{R}^{+}$
and $z=\frac{k}{n}-\textrm{i}k\varTheta$ with $-\vartheta_{1}\leq\varTheta\leq\vartheta_{2},$
where $k\ll\sqrt{n}$ and $k\vartheta_{1},k\vartheta_{2}\asymp\sqrt{\frac{1}{n}}.$
Then we have

\[
b_{j}\left(rk;z\right)\ll_{r}\left|z\right|^{-2j}.
\]
\end{lem}
The following lemma gives a representation of the I-Bessel function
as a integral.
\begin{lem}
\label{l2.3}\cite[Lemma 2.1]{BB}Suppose that $v\in\mathbb{R}$ and
$A,B\in\mathbb{R}^{+}$ satisfy $k\ll\sqrt{n},$ $A\ll\frac{n}{k}$,
$B\ll\frac{1}{k}$ and $k\vartheta_{1},k\vartheta_{2}\asymp\sqrt{\frac{1}{n}}.$
Then we have
\[
\int_{\frac{k}{n}-\textrm{i}k\vartheta_{1}}^{\frac{k}{n}+\textrm{i}k\vartheta_{2}}z^{-v}e^{Az+\frac{B}{z}}dz=2\pi\textrm{i}\left(\frac{A}{B}\right)^{\frac{v-1}{2}}I_{v-1}\left(2\sqrt{AB}\right)+\begin{cases}
O\left(n^{v-\frac{1}{2}}\right) & \textrm{if }v\geq0,\\
O\left(n^{\frac{v-1}{2}}\right) & \textrm{if }v<0.
\end{cases}
\]
\end{lem}

\subsection{\label{subsec:2.2}Modular and Jacobi forms}

Define $C^{*}\left(u;\tau\right)\coloneqq q^{-1/24}C\left(\zeta;q\right).$
The classical Jacobi theta function, defined by
\[
\begin{aligned}\vartheta\left(u\right)=\vartheta\left(u;\tau\right) & \coloneqq\textrm{i}\sum_{m\in\mathbb{Z}+\frac{1}{2}}\left(-1\right)^{m-\frac{1}{2}}q^{\frac{m^{2}}{2}}\zeta^{m}\\
 & =\sum_{m\in\mathbb{Z}+\frac{1}{2}}e^{\pi\textrm{i}m^{2}\tau+2\pi\textrm{i}m\left(u+\frac{1}{2}\right)},
\end{aligned}
\]
satisfies \cite{ZS}
\begin{equation}
\vartheta\left(u+\tau\right)=-e^{\pi\textrm{i}\tau-2\pi\textrm{i}u}\vartheta\left(u\right),\label{eq:2.1}
\end{equation}
and
\begin{equation}
\vartheta\left(-u\right)=-\vartheta\left(u\right).\label{eq:2.2}
\end{equation}

With $\tau\coloneqq\frac{1}{k}\left(h+\textrm{i}z\right),$ define
modular transforms:
\[
\tau_{1}^{*}\coloneqq\frac{1}{k}\left(\left[-h\right]_{k}^{*}+\frac{\textrm{i}}{z}\right),\quad\tau_{2}^{*}\coloneqq\frac{1}{k}\left(\left[-2h\right]_{k}^{*}+\frac{\textrm{i}}{2z}\right),\quad\tau_{3}^{*}\coloneqq\frac{1}{k}\left(\left[-h\right]_{k/2}^{*}+\frac{\textrm{i}}{z}\right).
\]
By \cite{KM}, we have the following transformation laws.
\begin{lem}
\label{l2.4} For $\textrm{Re}\left(z\right)>0$ and $\gcd\left(h,k\right)=1,$
we have
\[
\vartheta\left(u;\tau\right)=\chi\left(\begin{array}{cc}
\left[-h\right]_{k}^{*} & -\frac{h\left[-h\right]_{k}^{*}+1}{k}\\
k & -h
\end{array}\right)^{-3}\frac{1}{\sqrt{\textrm{i}z}}e^{-\frac{\pi ku^{2}}{z}}\vartheta\left(\frac{u}{\textrm{i}z};\tau_{1}^{*}\right),
\]
\[
\eta\left(\tau\right)=\chi\left(\begin{array}{cc}
\left[-h\right]_{k}^{*} & -\frac{h\left[-h\right]_{k}^{*}+1}{k}\\
k & -h
\end{array}\right)^{-1}\frac{1}{\sqrt{\textrm{i}z}}\eta\left(\tau_{1}^{*}\right),
\]
and
\[
C^{*}\left(u;\tau\right)=\frac{\sin\left(\pi u\right)}{\sin\left(\frac{\pi u}{\textrm{i}z}\right)}\chi\left(\begin{array}{cc}
\left[-h\right]_{k}^{*} & -\frac{h\left[-h\right]_{k}^{*}+1}{k}\\
k & -h
\end{array}\right)^{-1}\frac{1}{\sqrt{\textrm{i}z}}e^{\frac{\pi ku^{2}}{z}}C^{*}\left(\frac{u}{\textrm{i}z};\tau_{1}^{*}\right).
\]
\end{lem}
\begin{lem}
\label{l2.5} Suppose that $\textrm{Re}\left(z\right)>0$ and $\gcd\left(h,k\right)=1.$

(1) For odd $k,$ we have

\[
\vartheta\left(2u;2\tau\right)=\chi\left(\begin{array}{cc}
\left[-2h\right]_{k}^{*} & -\frac{2h\left[-2h\right]_{k}^{*}+1}{k}\\
k & -2h
\end{array}\right)^{-3}\frac{e^{-\frac{2\pi ku^{2}}{z}}}{\sqrt{2\textrm{i}z}}\vartheta\left(\frac{u}{\textrm{i}z};\tau_{2}^{*}\right),
\]
\[
\eta\left(2\tau\right)=\chi\left(\begin{array}{cc}
\left[-2h\right]_{k}^{*} & -\frac{2h\left[-2h\right]_{k}^{*}+1}{k}\\
k & -2h
\end{array}\right)^{-1}\frac{1}{\sqrt{2\textrm{i}z}}\eta\left(\tau_{2}^{*}\right),
\]

\[
C^{*}\left(u;2\tau\right)=\frac{\sin\left(\pi u\right)}{\sin\left(\frac{\pi u}{2\textrm{i}z}\right)}\chi\left(\begin{array}{cc}
\left[-2h\right]_{k}^{*} & -\frac{2h\left[-2h\right]_{k}^{*}+1}{k}\\
k & -2h
\end{array}\right)^{-1}\frac{e^{\frac{\pi ku^{2}}{2z}}}{\sqrt{2\textrm{i}z}}C^{*}\left(\frac{u}{2\textrm{i}z};\tau_{2}^{*}\right).
\]

(2) For even $k,$ we have
\[
\vartheta\left(2u;2\tau\right)=\chi\left(\begin{array}{cc}
\left[-h\right]_{k/2}^{*} & -2\frac{h\left[-h\right]_{k/2}^{*}+1}{k}\\
k/2 & -h
\end{array}\right)^{-3}\frac{e^{-\frac{2\pi ku^{2}}{z}}}{\sqrt{\textrm{i}z}}\vartheta\left(\frac{2u}{\textrm{i}z};2\tau_{3}^{*}\right),
\]
\[
\eta\left(2\tau\right)=\chi\left(\begin{array}{cc}
\left[-h\right]_{k/2}^{*} & -2\frac{h\left[-h\right]_{k/2}^{*}+1}{k}\\
k/2 & -h
\end{array}\right)^{-1}\frac{1}{\sqrt{\textrm{i}z}}\eta\left(2\tau_{3}^{*}\right),
\]

\[
C^{*}\left(u;2\tau\right)=\frac{\sin\left(\pi u\right)}{\sin\left(\frac{\pi u}{\textrm{i}z}\right)}\chi\left(\begin{array}{cc}
\left[-h\right]_{k/2}^{*} & -2\frac{h\left[-h\right]_{k/2}^{*}+1}{k}\\
k/2 & -h
\end{array}\right)^{-1}\frac{e^{\frac{\pi ku^{2}}{2z}}}{\sqrt{\textrm{i}z}}C^{*}\left(\frac{u}{\textrm{i}z};2\tau_{3}^{*}\right).
\]
\end{lem}

\subsection{\label{subsec:2.3}False theta functions}

False theta functions are series that are similar in form to those
of classical theta functions, but have different sign factors that
prevent them from being modular forms. We consider the false theta
function

\[
\psi\left(u;\tau\right)\coloneqq\textrm{i}\sum_{m\in\mathbb{Z}+\frac{1}{2}}\textrm{sgn}\left(m+\textrm{Im}\left(\frac{u}{\tau}\right)\right)\left(-1\right)^{m-\frac{1}{2}}q^{\frac{m^{2}}{2}}\zeta^{m},
\]
which lacks modularity due to the sign factor. Bringmann and Nazaroglu
\cite{BN} defined, for $\omega\in\mathbb{H},$ the completion of
$\psi:$
\[
\hat{\psi}\left(u;\tau,\omega\right)\coloneqq\textrm{i}\sum_{m\in\mathbb{Z}+\frac{1}{2}}\textrm{erf}\left(-\textrm{i}\sqrt{\pi\textrm{i}\left(\omega-\tau\right)}\left(m+\textrm{Im}\left(\frac{u}{\tau}\right)\right)\right)\left(-1\right)^{m-\frac{1}{2}}q^{\frac{m^{2}}{2}}\zeta^{m},
\]
where $\textrm{erf}\left(x\right)\coloneqq\frac{2}{\sqrt{\pi}}\int_{0}^{x}e^{-t^{2}}dt$
is the error function. Then we have
\[
\lim_{t\rightarrow\infty}\hat{\psi}\left(u;\tau,\tau+\textrm{i}t+\varepsilon\right)=\psi\left(u;\tau\right).
\]
Since $\hat{\psi}$ transforms like a Jacobi form, we derive a transformation
law for $\psi$ requiring the Eichler integrals:
\[
\begin{aligned}\mathcal{E}_{\frac{a}{c}}\left(u;\tau\right) & \coloneqq e^{-\frac{\pi\textrm{i}a}{c}\left(\textrm{Im}\left(\frac{u}{\tau}\right)\right)^{2}}\int_{\frac{a}{c}}^{\tau+\textrm{i}\infty+\varepsilon}e^{\pi\textrm{i}\delta\left(\textrm{Im}\left(\frac{u}{\tau}\right)\right)^{2}}\\
 & \times\frac{\sum_{m\in\mathbb{Z}+\frac{1}{2}}\left(m+\textrm{Im}\left(\frac{u}{\tau}\right)\right)\left(-1\right)^{m-\frac{1}{2}}q^{\frac{m^{2}}{2}}\zeta^{m}e^{\pi\textrm{i}\left(\delta m^{2}+2m\left(u+\left(\delta-\tau\right)\textrm{Im}\left(\frac{u}{\tau}\right)\right)\right)}}{\sqrt{\textrm{i}\left(\delta-\tau\right)}}d\delta.
\end{aligned}
\]
Thus we have the following lemma.
\begin{lem}
\label{l2.6} \cite{BN} For $u\in\mathbb{C}$ and $\left(\begin{array}{cc}
a & b\\
c & d
\end{array}\right)\in\textrm{SL}_{2}\left(\mathbb{Z}\right)$ with $c>0,$ we have
\[
\begin{aligned}\psi\left(u;\tau\right) & =\chi\left(\begin{array}{cc}
a & b\\
c & d
\end{array}\right)^{-3}\left(c\tau+d\right)^{-1/2}e^{-\frac{\pi\textrm{i}cu^{2}}{c\tau+d}}\left(\psi\left(\frac{u}{c\tau+d};\frac{a\tau+b}{c\tau+d}\right)\right.\\
 & \left.-\textrm{i}e^{\frac{\pi\textrm{i}}{c\left(c\tau+d\right)}\left(\frac{\textrm{Im}\left(u/\left(c\tau+d\right)\right)}{\textrm{Im}\left(\left(a\tau+b\right)/\left(c\tau+d\right)\right)}\right)^{2}}\mathcal{E}_{\frac{a}{c}}\left(\frac{u}{c\tau+d};\frac{a\tau+b}{c\tau+d}\right)\right).
\end{aligned}
\]
\end{lem}

\subsection{Asymptotic Tools}

A multivariable function $f$ in $\ell$ variables is of sufficient
decay in $D$ if there exist $\varepsilon_{1},\ldots,\varepsilon_{\ell}>0$
such that $f\left(x_{1},\ldots,x_{\ell}\right)\ll\left(x_{1}+1\right)^{-1-\varepsilon_{1}}\cdot\cdot\cdot\left(x_{\ell}+1\right)^{-1-\varepsilon_{\ell}}$
uniformly as $\left|x_{1}\right|+...+\left|x_{\ell}\right|\rightarrow\infty$
in $D.$ Also, we define $B_{r}\left(x\right)$ as the $r$-th Bernoulli
polynomial and $C_{R}\left(0\right)$ as the circle around 0 with
radius $R.$
\begin{prop}
\label{p2.1} (Euler-Maclaurin summation formula) Let $D_{\theta}\coloneqq\left\{ re^{\textrm{i}\alpha}:r\geq0\textrm{ and }\left|\alpha\right|\leq\theta\right\} $
with $0\leq\theta<\frac{\pi}{2}.$ Let $f:\mathbb{C}\rightarrow\mathbb{C}$
be holomorphic in a domain containing $D_{\theta},$ so that in particular
$f$ is holomorphic at the origin, and assume that $z\rightarrow f\left(z\right)$
and all of its derivatives are of sufficient decay. Then for $a\in\mathbb{R}$
and $N\in\mathbb{N}_{0},$ we have, uniformly as $z\rightarrow0$
in $D_{\theta},$
\[
\begin{aligned}\sum_{m\geq0}f\left(\left(m+a\right)z\right) & =\frac{1}{z}\int_{0}^{\infty}f\left(\omega\right)d\omega-\sum_{n=0}^{N-1}\frac{B_{n+1}\left(a\right)f^{\left(n\right)}\left(0\right)}{\left(n+1\right)!}z^{n}+O\left(z^{N}\right).\end{aligned}
\]
\end{prop}
\begin{prop}
\label{p2.2}\cite[Theorem B.5 ]{MV} For $N\in\mathbb{N}$ and continuously
differentiable $g:\mathbb{R}\rightarrow\mathbb{C},$ we have
\[
\begin{aligned}\sum_{k=1}^{N}g\left(k\right) & =\int_{1}^{N}g\left(u\right)du+\frac{1}{2}\left(g\left(N\right)+g\left(1\right)\right)+\int_{1}^{N}\left(\left\{ u\right\} -\frac{1}{2}\right)g^{\prime}\left(u\right)du\\
 & =\int_{0}^{N}g\left(u\right)du+\frac{1}{2}\left(g\left(N\right)-g\left(0\right)\right)+\int_{0}^{N}\left(\left\{ u\right\} -\frac{1}{2}\right)g^{\prime}\left(u\right)du.
\end{aligned}
\]
\end{prop}
The following lemma is useful for the approximation logarithmic series,
which can be deduced by Taylor expansions.
\begin{lem}
\label{l2.7}\cite[Lemma 2.2]{BB2}There exists a constant $C$ such
that for all $0<x<1$ and $s\in\mathbb{R},$ we have 
\[
\left|\textrm{Log}\left(\frac{1\pm x}{1\pm xe^{\textrm{i}s}}\right)-\frac{\textrm{i}sx}{1\pm x}+\frac{s^{2}x}{2\left(1\pm x\right)^{2}}\right|\leq C\frac{x\left|s\right|^{3}}{\left(1-x\right)^{3}}.
\]
\end{lem}
The following lemma is about the asymptotic behavior of a certain
product, which is compared with a similar formula in \cite[(6.10)]{FB}
and \cite[Lemma 2.3]{BB2}.
\begin{lem}
\label{l2.8} Uniformly in $v\geq-\frac{\log\left(n\right)}{8}$ as
$n\rightarrow\infty,$ we have 
\[
\prod_{2k-1>B\sqrt{n}\left(v+\log\left(B\sqrt{n}\right)\right)}\left(1-e^{-\frac{2k-1}{B\sqrt{n}}}\right)\sim e^{-\frac{1}{2}e^{-v}},
\]
where $B=\frac{\sqrt{6}}{\pi}.$
\end{lem}

\subsection{Saddle-point method}

We employ a specific variant of the saddle-point method for evaluating
Cauchy integrals. This approach closely follows the one used by Fristedt
\cite[Proposition 4.5]{FB} and appears also in the proofs of Proposition
3 in \cite{RP} and Proposition 3 in \cite{BW}.
\begin{prop}
\label{p2.3}\cite[Proposition 2.5]{BB2} Suppose that $\left\{ g_{n}\right\} _{n\geq1}$
is a sequence of twice continuously differentiable functions. For
all sufficiently small fixed $\varepsilon>0,$ after decomposing the
integral as
\[
\int_{-\frac{1}{2}}^{\frac{1}{2}}\exp\left(g_{n}\left(2\pi\textrm{i}\theta\right)\right)d\theta=\int_{-\varepsilon n^{\frac{1}{2}}}^{\varepsilon n^{\frac{1}{2}}}\exp\left(g_{n}\left(2\pi\textrm{i}\theta\right)\right)d\theta+\int_{\varepsilon n^{\frac{1}{2}}<\left|\theta\right|\leq\frac{1}{2}}\exp\left(g_{n}\left(2\pi\textrm{i}\theta\right)\right)d\theta,
\]
the following holds as $n\rightarrow\infty.$

(i) We can obtain $g_{n}\left(0\right)\asymp n^{\frac{1}{2}}$ and
$g_{n}^{\prime\prime}\left(0\right)\asymp n^{\frac{3}{2}},$ where
the implied constants are poistive real numbers, and also $g_{n}^{\prime}\left(0\right)=o\left(n^{\frac{3}{4}}\right).$

(ii) The major arc: for $\left|\theta\right|\leq\varepsilon n^{\frac{1}{2}},$
we have
\[
\left|g_{n}\left(2\pi\textrm{i}\theta\right)-g_{n}\left(0\right)-g_{n}^{\prime}\left(0\right)2\pi\textrm{i}\theta-g_{n}^{\prime\prime}\left(0\right)\frac{\left(2\pi\textrm{i}\theta\right)^{2}}{2}\right|=O\left(\theta^{3}n^{2}\right).
\]

(iii) The minor arc: for some $\delta_{\varepsilon}>0$ and $\varepsilon n^{\frac{1}{2}}<\left|\theta\right|\leq\frac{1}{2},$
we have
\[
\limsup_{n\rightarrow\infty}\frac{\textrm{Re}\left(g_{n}\left(2\pi\textrm{i}\theta\right)\right)-g_{n}\left(0\right)}{n^{\frac{1}{2}}}<-\delta_{\varepsilon}.
\]

Then we have
\[
\int_{-\frac{1}{2}}^{\frac{1}{2}}\exp\left(g_{n}\left(2\pi\textrm{i}\theta\right)\right)d\theta\sim\frac{e^{g_{n}\left(0\right)}}{\sqrt{2\pi g_{n}^{\prime\prime}\left(0\right)}}.
\]
\end{prop}

\subsection{Probability Tools}

Now we define the total variation metric $d_{TV},$ which is defined
on the measure $\mu$ and $\nu$ on $\mathbb{R}^{d}$ by 
\[
d_{TV}\left(\mu,\nu\right)\coloneqq\sup_{\textrm{Borel }B\subset\mathbb{R}^{d}}\left(\mu\left(B\right)-\nu\left(B\right)\right).
\]
We also use the celebrated Chebyshev inequality.
\begin{thm}
\textup{(Chebyshev\textquoteright s inequality)} Let $X$ be a square-integrable
random variable under a probability measure $\boldsymbol{P}$ with
finite expectation $m$ and variance $\sigma.$ Then for any $t>0,$
we have
\begin{equation}
\boldsymbol{P}\left(\left|X-m\right|\geq t\right)\leq\frac{\sigma^{2}}{t^{2}}.\label{eq:2.3}
\end{equation}
\end{thm}

\section{\label{sec:3} Rank Moment Asymptotics}

In this section, we prove Theorem \ref{t1.1}, Corollary \ref{c1.1}
and Theorem \ref{t1.2}

\subsection{Generating Function Decomposition}

Define 
\[
\begin{aligned}\textrm{OU}_{1}\left(u;\tau\right)\coloneqq & -\frac{\textrm{i}}{2}\frac{C^{*}\left(u;\tau\right)}{\eta\left(\tau\right)}\frac{\eta\left(2\tau\right)}{C^{*}\left(u;2\tau\right)}\left(\vartheta\left(2u;2\tau\right)+\psi\left(2u;2\tau\right)\right),\end{aligned}
\]
and 
\[
\begin{aligned}H_{1}\left(\zeta;q\right)\coloneqq & \sum_{n\geq0}\left(-1\right)^{n+1}\zeta^{3n+1}q^{3n^{2}+2n}\left(1+\zeta q^{2n+1}\right).\end{aligned}
\]

\begin{lem}
\label{l3.1}Suppose that $u\in\mathbb{R}$ and $\ell\in2\mathbb{N}_{0}.$
Then we have

\[
\left[\frac{\partial^{\ell}}{\partial u^{\ell}}\textrm{OU}\left(\zeta;q\right)\right]_{u=0}=\left[\frac{\partial^{\ell}}{\partial u^{\ell}}\left[q^{-\frac{1}{4}}\textrm{OU}_{1}\left(u;\tau\right)+H_{1}\left(\zeta;q\right)\right]\right]_{u=0}.
\]
\end{lem}
\begin{proof}
It is obvious that $\textrm{sgn}\left(m\right)=\textrm{sgn}\left(m+\frac{\textrm{Im}\left(u\right)}{\textrm{Im}\left(\tau\right)}\right)$
for $u\in\mathbb{R}.$ By \eqref{eq:1.3} and \cite[Lemma 3.1]{BB},
we can obtain the required result easily.
\end{proof}
Using Lemmas \ref{l2.4}, \ref{l2.5} and \ref{l2.6}, we can easily
obtain the transformations for the function $\textrm{OU}_{1}.$
\begin{thm}
\label{t3.1}Suppose that $z\in\mathbb{C}$ with $\textrm{Re}\left(z\right)>0,$
$0\leq h<k,$ $\gcd\left(h,k\right)=1,$ and $u\in\mathbb{R}$ with
$\left|ku\right|<\frac{1}{4}.$

(1) For odd $k,$ we have
\begin{equation}
\begin{aligned}\textrm{OU}_{1}\left(u;\tau\right) & =\chi_{2h,k}\left(2\textrm{i}z\right)^{-1/2}\left[f_{3k}\left(u;-z\right)\textrm{OU}_{1}\left(\frac{u}{\textrm{i}z};\tau_{2}^{*}\right)\right.\\
 & \left.-\frac{1}{2}\frac{C^{*}\left(\frac{u}{\textrm{i}z};\tau_{1}^{*}\right)}{\eta\left(\tau_{1}^{*}\right)}\frac{\eta\left(\tau_{2}^{*}\right)}{C^{*}\left(\frac{u}{2\textrm{i}z};\tau_{2}^{*}\right)}f_{k}\left(u;-z\right)\mathcal{E}_{\frac{\left[-2h\right]_{k}}{k}}\left(\frac{u}{\textrm{i}z};\tau_{2}^{*}\right)\right].
\end{aligned}
\label{eq:3.1}
\end{equation}

(2) For even $k,$ we have
\begin{equation}
\begin{aligned}\textrm{OU}_{1}\left(u;\tau\right) & =\chi_{h,k/2}\left(\textrm{i}z\right)^{-1/2}\left[e^{-\frac{3\pi ku^{2}}{2z}}\textrm{OU}_{1}\left(\frac{2u}{\textrm{i}z};2\tau_{3}^{*}\right)\right.\\
 & \left.-\frac{1}{2}\frac{C^{*}\left(\frac{u}{\textrm{i}z};\tau_{1}^{*}\right)}{\eta\left(\tau_{1}^{*}\right)}\frac{\eta\left(2\tau_{3}^{*}\right)}{C^{*}\left(\frac{u}{\textrm{i}z};2\tau_{3}^{*}\right)}e^{\frac{\pi ku^{2}}{2z}}\mathcal{E}_{\frac{2\left[-h\right]_{k/2}}{k}}\left(\frac{2u}{\textrm{i}z};2\tau_{3}^{*}\right)\right].
\end{aligned}
\label{eq:3.2}
\end{equation}
\end{thm}

\subsection{Mordell Integral Representations}

Proceeding as in \cite[Lemma 4.1]{BB}, we obtain the following representations
of $\mathcal{E}_{\frac{\left[-2h\right]_{k}^{*}}{k}}$ and $\mathcal{E}_{\frac{2\left[-h\right]_{k/2}^{*}}{k}}.$
\begin{lem}
Suppose that $u\in\mathbb{R}$ is sufficiently small and $\textrm{Re}\left(z\right)>0.$

(1) For odd $k,$ we have

\[
\begin{aligned}\mathcal{E}_{\frac{\left[-2h\right]_{k}^{*}}{k}}\left(\frac{u}{\textrm{i}z};\tau_{2}^{*}\right) & =\frac{1}{\pi\textrm{i}}\sum_{m\in\mathbb{Z}+\frac{1}{2}}\left(-1\right)^{m-\frac{1}{2}}e^{\pi\textrm{i}m^{2}\frac{\left[-2h\right]_{k}^{*}}{k}}\\
 & \times\lim_{\varepsilon\rightarrow0^{+}}\int_{-\infty}^{\infty}\frac{e^{-\frac{\pi x^{2}}{2kz}}}{x-\left(m-2ku\right)\left(1+\textrm{i}\varepsilon\right)}dx.
\end{aligned}
\]
For $0\leq D_{1}\leq\frac{1}{12},$ we rewrite
\begin{equation}
e^{\frac{\pi D_{1}}{kz}}\mathcal{E}_{\frac{\left[-2h\right]_{k}^{*}}{k}}\left(\frac{u}{\textrm{i}z};\tau_{2}^{*}\right)=\mathcal{E}_{\frac{\left[-2h\right]_{k}^{*}}{k},D_{1}}^{*}\left(\frac{u}{\textrm{i}z};\tau_{2}^{*}\right)+\mathcal{E}_{\frac{\left[-2h\right]_{k}^{*}}{k},D_{1}}^{e}\left(\frac{u}{\textrm{i}z};\tau_{2}^{*}\right),\label{eq:3.3}
\end{equation}
where
\[
\begin{aligned}\mathcal{E}_{\frac{\left[-2h\right]_{k}^{*}}{k},D_{1}}^{*}\left(\frac{u}{\textrm{i}z};\tau_{2}^{*}\right) & \coloneqq\frac{e^{\frac{\pi D_{1}}{kz}}}{\pi\textrm{i}}\sum_{m\in\mathbb{Z}+\frac{1}{2}}\left(-1\right)^{m-\frac{1}{2}}e^{\pi\textrm{i}m^{2}\frac{\left[-2h\right]_{k}^{*}}{k}}\\
 & \times\lim_{\varepsilon\rightarrow0^{+}}\int_{-\sqrt{2D_{1}}}^{\sqrt{2D_{1}}}\frac{e^{-\frac{\pi x^{2}}{2kz}}}{x-\left(m-2ku\right)\left(1+\textrm{i}\varepsilon\right)}dx,
\end{aligned}
\]
and
\[
\begin{aligned}\mathcal{E}_{\frac{\left[-2h\right]_{k}^{*}}{k},D_{1}}^{e}\left(\frac{u}{\textrm{i}z};\tau_{2}^{*}\right) & \coloneqq\frac{e^{\frac{\pi D_{1}}{kz}}}{\pi\textrm{i}}\sum_{m\in\mathbb{Z}+\frac{1}{2}}\left(-1\right)^{m-\frac{1}{2}}e^{\pi\textrm{i}m^{2}\frac{\left[-2h\right]_{k}^{*}}{k}}\\
 & \times\lim_{\varepsilon\rightarrow0^{+}}\int_{\left|x\right|\geq\sqrt{2D_{1}}}\frac{e^{-\frac{\pi x^{2}}{2kz}}}{x-\left(m-2ku\right)\left(1+\textrm{i}\varepsilon\right)}dx.
\end{aligned}
\]

(2) For even $k,$ we have
\[
\begin{aligned}\mathcal{E}_{\frac{2\left[-h\right]_{k/2}^{*}}{k}}\left(\frac{2u}{\textrm{i}z};2\tau_{3}^{*}\right) & =\frac{1}{\pi\textrm{i}}\sum_{m\in\mathbb{Z}+\frac{1}{2}}\left(-1\right)^{m-\frac{1}{2}}e^{\pi\textrm{i}m^{2}\frac{2\left[-h\right]_{k/2}^{*}}{k}}\\
 & \times\lim_{\varepsilon\rightarrow0^{+}}\int_{-\infty}^{\infty}\frac{e^{-\frac{2\pi x^{2}}{kz}}}{x-\left(m-ku\right)\left(1+\textrm{i}\varepsilon\right)}dx.
\end{aligned}
\]
For $0\leq D_{2}\leq\frac{1}{24},$ we rewrite
\begin{equation}
e^{-\frac{4\pi D_{2}}{kz}}\mathcal{E}_{\frac{2\left[-h\right]_{k/2}^{*}}{k}}\left(\frac{2u}{\textrm{i}z};2\tau_{3}^{*}\right)=\mathcal{E}_{\frac{2\left[-h\right]_{k/2}^{*}}{k},D_{2}}^{*}\left(\frac{2u}{\textrm{i}z};2\tau_{3}^{*}\right)+\mathcal{E}_{\frac{2\left[-h\right]_{k/2}^{*}}{k},D_{2}}^{e}\left(\frac{2u}{\textrm{i}z};2\tau_{3}^{*}\right),\label{eq:3.4}
\end{equation}
where
\[
\begin{aligned}\mathcal{E}_{\frac{2\left[-h\right]_{k/2}^{*}}{k},D_{2}}^{*}\left(\frac{2u}{\textrm{i}z};2\tau_{3}^{*}\right) & \coloneqq\frac{e^{-\frac{4\pi D_{2}}{kz}}}{\pi\textrm{i}}\sum_{m\in\mathbb{Z}+\frac{1}{2}}\left(-1\right)^{m-\frac{1}{2}}e^{\pi\textrm{i}m^{2}\frac{2\left[-h\right]_{k/2}^{*}}{k}}\\
 & \times\lim_{\varepsilon\rightarrow0^{+}}\int_{-\sqrt{2D_{2}}}^{\sqrt{2D_{2}}}\frac{e^{-\frac{2\pi x^{2}}{kz}}}{x-\left(m-ku\right)\left(1+\textrm{i}\varepsilon\right)}dx,
\end{aligned}
\]
and
\[
\begin{aligned}\mathcal{E}_{\frac{2\left[-h\right]_{k/2}^{*}}{k},D_{2}}^{e}\left(\frac{2u}{\textrm{i}z};2\tau_{3}^{*}\right) & \coloneqq\frac{e^{-\frac{4\pi D_{2}}{kz}}}{\pi\textrm{i}}\sum_{m\in\mathbb{Z}+\frac{1}{2}}\left(-1\right)^{m-\frac{1}{2}}e^{\pi\textrm{i}m^{2}\frac{2\left[-h\right]_{k/2}^{*}}{k}}\\
 & \times\lim_{\varepsilon\rightarrow0^{+}}\int_{\left|x\right|\geq\sqrt{2D_{2}}}\frac{e^{-\frac{2\pi x^{2}}{kz}}}{x-\left(m-ku\right)\left(1+\textrm{i}\varepsilon\right)}dx.
\end{aligned}
\]
\end{lem}
By \cite[Lemma 3.3]{BN} and \cite[Lemma 4.2]{BB}, we can derive
the following lemma.
\begin{lem}
\label{l3.3}Suppose that $\ell\in2\mathbb{N}_{0}.$

(1) For odd $k$ and $0\leq D_{1}\leq\frac{1}{12},$ we have
\[
\left[\frac{\partial^{\ell}}{\partial u^{\ell}}\mathcal{E}_{\frac{\left[-2h\right]_{k}^{*}}{k},D_{1}}^{e}\left(\frac{u}{\textrm{i}z};\tau_{2}^{*}\right)\right]_{u=0}\ll\log\left(k\right)+k^{\ell}.
\]

(2) For even $k$ and $0\leq D_{2}\leq\frac{1}{24},$ we have
\[
\left[\frac{\partial^{\ell}}{\partial u^{\ell}}\mathcal{E}_{\frac{2\left[-h\right]_{k/2}^{*}}{k},D_{2}}^{e}\left(\frac{2u}{\textrm{i}z};2\tau_{3}^{*}\right)\right]_{u=0}\ll\log\left(k\right)+k^{\ell}.
\]
\end{lem}
Using \cite[(3.11)]{BN} and \cite[Lemma 4.3]{BB}, we can obtain
the following results involving a Mordell type integral.
\begin{lem}
\label{l3.4}Suppose that $u\in\mathbb{R}$ is sufficiently small
and $\textrm{Re}\left(z\right)>0.$

(1) For odd $k,$ we have
\[
\begin{aligned}\mathcal{E}_{\frac{\left[-2h\right]_{k}^{*}}{k},\frac{1}{12}}^{*}\left(\frac{u}{\textrm{i}z};\tau_{2}^{*}\right) & =\frac{1}{4\sqrt{3}\pi\textrm{i}}\sum_{v=0}^{2k-1}\left(-1\right)^{v}e^{\pi\textrm{i}\left(v+\frac{1}{2}\right)^{2}\frac{\left[-2h\right]_{k}^{*}}{k}}\\
 & \times\int_{-1}^{1}\cot\left(\frac{\pi}{2k}\left(\frac{x}{2\sqrt{3}}-v-\frac{1}{2}+2ku\right)\right)e^{\frac{\pi}{12kz}\left(1-x^{2}\right)}dx.
\end{aligned}
\]

(2) For even $k,$ we have
\[
\begin{aligned}\mathcal{E}_{\frac{2\left[-h\right]_{k/2}^{*}}{k},\frac{1}{24}}^{*}\left(\frac{2u}{\textrm{i}z};2\tau_{3}^{*}\right) & =\frac{1}{2\sqrt{6}\pi\textrm{i}}\sum_{v=0}^{k-1}\left(-1\right)^{v}e^{\pi\textrm{i}\left(v+\frac{1}{2}\right)^{2}\frac{2\left[-h\right]_{k/2}^{*}}{k}}\\
 & \times\int_{-1}^{1}\cot\left(\frac{\pi}{k}\left(\frac{x}{\sqrt{6}}-v-\frac{1}{2}+ku\right)\right)e^{-\frac{\pi}{6kz}\left(1+x^{2}\right)}dx.
\end{aligned}
\]
\end{lem}

\subsection{Taylor Coefficient Asymptotics}
\begin{thm}
\label{t3.2}Suppose that $\ell\in2\mathbb{N}_{0}$ and $z\in\mathbb{C}$
with $\textrm{Re}\left(\frac{1}{z}\right)\geq\frac{k}{2}$ and $\left|z\right|\ll\frac{1}{k}.$

(1) For odd $k,$ we have
\begin{equation}
\begin{aligned}\left[\frac{\partial^{\ell}}{\partial u^{\ell}}\textrm{OU}_{1}\left(u;\tau\right)\right]_{u=0} & =\frac{\left(2\pi\textrm{i}\right)^{\ell}\textrm{i}^{\frac{1}{2}}}{16\sqrt{6}k}\chi_{2h,k}\sum_{v=0}^{2k-1}\left(-1\right)^{v}e^{\pi\textrm{i}\left(v^{2}+v\right)\frac{\left[-2h\right]_{k}^{*}}{k}+\frac{\pi\textrm{i}}{12k}\left(5\left[-2h\right]_{k}^{*}-2\left[-h\right]_{k}^{*}\right)}\\
 & \times\underset{\substack{0\leq j\leq\ell/2}
}{\mathop{\sum}}\left(\begin{array}{c}
\ell\\
2j
\end{array}\right)\left(-\frac{1}{4}\right)^{j}\underset{\substack{a+b=j\\
a,b\geq0
}
}{\mathop{\sum}}k^{a}\kappa\left(a,b\right)z^{-\frac{1}{2}-a-2b}\\
 & \times\int_{-1}^{1}C_{\ell-2j}\left(\frac{1}{2k}\left(\frac{x}{2\sqrt{3}}-v-\frac{1}{2}\right)\right)e^{\frac{\pi}{12kz}\left(1-x^{^{2}}\right)}dx\\
 & +O\left(\log\left(k\right)\left|z\right|^{-\frac{1}{2}-\ell}\right).
\end{aligned}
\label{eq:3.5}
\end{equation}

(2) For even $k,$ we have
\begin{equation}
\begin{aligned}\left[\frac{\partial^{\ell}}{\partial u^{\ell}}\textrm{OU}_{1}\left(u;\tau\right)\right]_{u=0} & =\frac{\left(2\pi\textrm{i}\right)^{\ell}\textrm{i}^{\frac{1}{2}}}{4\sqrt{6}k}\chi_{h,k/2}\sum_{v=0}^{k-1}\left(-1\right)^{v}e^{\pi\textrm{i}\left(v^{2}+v\right)\frac{2\left[-h\right]_{k/2}^{*}}{k}+\frac{\pi\textrm{i}}{6k}\left(5\left[-h\right]_{k/2}^{*}-\left[-h\right]_{k}^{*}\right)}\\
 & \times\underset{\substack{0\leq j\leq\ell/2}
}{\mathop{\sum}}\left(\begin{array}{c}
\ell\\
2j
\end{array}\right)\frac{k^{j}\pi^{j}}{j!2^{j}}z^{-\frac{1}{2}-j}\int_{-1}^{1}C_{\ell-2j}\left(\frac{1}{k}\left(\frac{x}{\sqrt{6}}-v-\frac{1}{2}\right)\right)\\
 & \times e^{-\frac{\pi}{6kz}\left(1+x^{^{2}}\right)}dx+O\left(\log\left(k\right)\left|z\right|^{-\frac{1}{2}-\ell}\right).
\end{aligned}
\label{eq:3.6}
\end{equation}
\end{thm}
\begin{proof}
Define 
\[
\textrm{OU}_{1}\left(u;\tau\right)\coloneqq q^{1/4}\sum_{j\geq0}a_{j}\left(\tau\right)\frac{\left(2\pi\textrm{i}u\right)^{j}}{j!},
\]
where $\left|a_{j}\left(\tau\right)\right|\ll e^{2\pi\textrm{i}\tau}$
as $\tau\rightarrow\textrm{i}\infty.$

First, we consider the case that $k$ is odd. For the first term in
\eqref{eq:3.1}, applying Lemma \ref{l2.1}, we can obtain that
\[
\begin{aligned}f_{3k}\left(u;-z\right)\textrm{OU}_{1}\left(\frac{u}{\textrm{i}z};\tau_{2}^{*}\right) & =\frac{1}{2}e^{\frac{\pi\textrm{i}\tau_{2}^{*}}{2}}\sum_{\ell\geq0}\frac{\left(2\pi\textrm{i}u\right)^{\ell}}{\ell!}\\
 & \quad\times\left[\underset{\substack{2r+j=\ell\\
r,j\geq0
}
}{\mathop{\sum}}\frac{\ell!}{\left(2r\right)!j!}\left(-\frac{1}{4}\right)^{r}b_{r}\left(3k;-z\right)a_{j}\left(\tau_{2}^{*}\right)\left(-z\right)^{-j}\right].
\end{aligned}
\]
It is easy to see that $\left|a_{j}\left(\tau_{2}^{*}\right)\right|\ll1$
as $z\rightarrow0$ with $\textrm{Re}\left(z\right)>0.$ Thus, by
Lemma \ref{l2.2},
\begin{equation}
\left[\frac{\partial^{\ell}}{\partial u^{\ell}}f_{3k}\left(u;-z\right)\textrm{OU}_{1}\left(\frac{u}{\textrm{i}z};\tau_{2}^{*}\right)\right]_{u=0}\ll\left|z\right|^{-\ell}e^{-\frac{\pi}{4k}\textrm{Re}\left(\frac{1}{z}\right)}.\label{eq:3.7}
\end{equation}
Consequently, the main term for odd $k$ comes from the second term
in \eqref{eq:3.1}.

Since 
\[
\frac{C^{*}\left(u;\tau\right)}{\eta\left(\tau\right)}=q^{-1/12}\frac{C^{*}\left(\frac{u}{\textrm{i}z};\tau_{1}^{*}\right)}{\eta\left(\tau_{1}^{*}\right)},
\]
we have
\begin{equation}
\begin{aligned}\frac{C^{*}\left(u;\tau\right)}{\eta\left(\tau\right)}\frac{\eta\left(2\tau\right)}{C^{*}\left(u;2\tau\right)} & =e^{-\frac{\pi\textrm{i}}{6k}\left(\left[-h\right]_{k}^{*}-\left[-2h\right]_{k}^{*}\right)+\frac{\pi}{12kz}}\frac{C^{*}\left(\frac{u}{\textrm{i}z};\tau_{1}^{*}\right)}{\eta\left(\tau_{1}^{*}\right)}\frac{\eta\left(\tau_{2}^{*}\right)}{C^{*}\left(\frac{u}{2\textrm{i}z};\tau_{2}^{*}\right)}\\
 & =e^{-\frac{\pi\textrm{i}}{6k}\left(\left[-h\right]_{k}^{*}-\left[-2h\right]_{k}^{*}\right)+\frac{\pi}{12kz}}\left(1+\sum_{j\geq0}\gamma_{j}\left(\tau_{1}^{*},\tau_{2}^{*}\right)\frac{\left(2\pi\textrm{i}u\right)^{j}}{j!}\right),
\end{aligned}
\label{eq:3.8}
\end{equation}
where $\left|\gamma_{j}\left(\tau_{1}^{*},\tau_{2}^{*}\right)\right|\ll e^{2\pi\textrm{i}\tau}$
as $\tau\rightarrow\textrm{i}\infty.$ Therefore, we can derive that
\[
\begin{aligned} & -\frac{1}{2}\chi_{2h,k}\left(2\textrm{i}z\right)^{-1/2}\frac{C^{*}\left(\frac{u}{\textrm{i}z};\tau_{1}^{*}\right)}{\eta\left(\tau_{1}^{*}\right)}\frac{\eta\left(\tau_{2}^{*}\right)}{C^{*}\left(\frac{u}{2\textrm{i}z};\tau_{2}^{*}\right)}f_{k}\left(u;-z\right)\mathcal{E}_{\frac{\left[-2h\right]_{k}^{*}}{k}}\left(\frac{u}{\textrm{i}z};\tau_{2}^{*}\right)\\
 & =-\frac{1}{2}\chi_{2h,k}\left(2\textrm{i}z\right)^{-1/2}f_{k}\left(u;-z\right)\left[e^{-\frac{\pi\textrm{i}}{6k}\left(\left[-h\right]_{k}^{*}-\left[-2h\right]_{k}^{*}\right)}\mathcal{E}_{\frac{\left[-2h\right]_{k}^{*}}{k},\frac{1}{12}}^{*}\left(\frac{u}{\textrm{i}z};\tau_{2}^{*}\right)\right.\\
 & +e^{-\frac{\pi\textrm{i}}{6k}\left(\left[-h\right]_{k}^{*}-\left[-2h\right]_{k}^{*}\right)}\mathcal{E}_{\frac{\left[-2h\right]_{k}^{*}}{k},\frac{1}{12}}^{e}\left(\frac{u}{\textrm{i}z};\tau_{2}^{*}\right)\\
 & \left.+\left(\frac{C^{*}\left(\frac{u}{\textrm{i}z};\tau_{1}^{*}\right)}{\eta\left(\tau_{1}^{*}\right)}\frac{\eta\left(\tau_{2}^{*}\right)}{C^{*}\left(\frac{u}{2\textrm{i}z};\tau_{2}^{*}\right)}-e^{-\frac{\pi\textrm{i}}{6k}\left(\left[-h\right]_{k}^{*}-\left[-2h\right]_{k}^{*}\right)}\right)\mathcal{E}_{\frac{\left[-2h\right]_{k}^{*}}{k},0}^{e}\left(\frac{u}{\textrm{i}z};\tau_{2}^{*}\right)\right].
\end{aligned}
\]
Applying Lemmas \ref{l2.2} and \ref{l3.3}, we have
\begin{equation}
\left[\frac{\partial^{\ell}}{\partial u^{\ell}}f_{k}\left(u;-z\right)\mathcal{E}_{\frac{\left[-2h\right]_{k}^{*}}{k},D_{1}}^{e}\left(\frac{u}{\textrm{i}z};\tau_{2}^{*}\right)\right]_{u=0}\ll\log\left(k\right)\left|z\right|^{-\ell},\label{eq:3.9}
\end{equation}
as $0\leq D_{1}\leq\frac{1}{12}.$ By \eqref{eq:3.7} and \eqref{eq:3.9},
we can obtain that
\begin{equation}
\begin{alignedat}{1}\left[\frac{\partial^{\ell}}{\partial u^{\ell}}\textrm{OU}_{1}\left(u;\tau\right)\right]_{u=0}= & -\frac{1}{2}\chi_{2h,k}\left(2\textrm{i}z\right)^{-1/2}e^{-\frac{\pi\textrm{i}}{6k}\left(\left[-h\right]_{k}^{*}-\left[-2h\right]_{k}^{*}\right)}\\
 & \times\left[\frac{\partial^{\ell}}{\partial u^{\ell}}\left(f_{k}\left(u;-z\right)\mathcal{E}_{\frac{\left[-2h\right]_{k}^{*}}{k},\frac{1}{12}}^{*}\left(\frac{u}{\textrm{i}z};\tau_{2}^{*}\right)\right)\right]_{u=0}\\
 & +O\left(\log\left(k\right)\left|z\right|^{-\frac{1}{2}-\ell}\right).
\end{alignedat}
\label{eq:3.10}
\end{equation}
Using Lemma \ref{l3.4}, we have
\[
\begin{aligned} & \left[\frac{\partial^{\ell}}{\partial u^{\ell}}\left(f_{k}\left(u;-z\right)\mathcal{E}_{\frac{\left[-2h\right]_{k}^{*}}{k},\frac{1}{12}}^{*}\left(\frac{u}{\textrm{i}z};\tau_{2}^{*}\right)\right)\right]_{u=0}=-\frac{\textrm{i}}{4\sqrt{3}\pi}\sum_{v=0}^{2k-1}\left(-1\right)^{v}e^{\pi\textrm{i}\left(v+\frac{1}{2}\right)^{2}\frac{\left[-2h\right]_{k}^{*}}{k}}\\
 & \times\int_{-1}^{1}\left[\frac{\partial^{\ell}}{\partial u^{\ell}}\left(f_{k}\left(u;-z\right)\cot\left(\frac{\pi}{2k}\left(\frac{x}{2\sqrt{3}}-v-\frac{1}{2}+2ku\right)\right)\right)\right]_{u=0}e^{\frac{\pi}{12kz}\left(1-x^{2}\right)}dx.
\end{aligned}
\]
Then we apply Lemma \ref{l2.1} to derive
\[
\begin{aligned} & \left[\frac{\partial^{\ell}}{\partial u^{\ell}}\left(f_{k}\left(u;-z\right)\cot\left(\frac{\pi}{2k}\left(\frac{x}{2\sqrt{3}}-v-\frac{1}{2}+2ku\right)\right)\right)\right]_{u=0}\\
 & =\frac{\left(2\pi\textrm{i}\right)^{\ell}}{2}\underset{\substack{0\leq j\leq\ell/2}
}{\mathop{\sum}}\left(\begin{array}{c}
\ell\\
2j
\end{array}\right)\left(-\frac{1}{4}\right)^{j}\underset{\substack{a+b=j\\
a,b\geq0
}
}{\mathop{\sum}}k^{a}\kappa\left(a,b\right)z^{-a-2b}C_{\ell-2j}\left(\frac{1}{2k}\left(\frac{x}{2\sqrt{3}}-v-\frac{1}{2}\right)\right).
\end{aligned}
\]
Substituting this into \eqref{eq:3.10} completes the proof of \eqref{eq:3.5}.

We now consider the case that $k$ is even. For the first term in
\eqref{eq:3.2}, we have 
\[
\begin{aligned}e^{-\frac{3\pi ku^{2}}{2z}}\textrm{OU}_{1}\left(\frac{2u}{\textrm{i}z};2\tau_{3}^{*}\right) & =\frac{1}{2}e^{\pi\textrm{i}\tau_{3}^{*}}\sum_{\ell\geq0}\frac{\left(2\pi\textrm{i}u\right)^{\ell}}{\ell!}\left[\underset{\substack{2r+j=\ell\\
r,j\geq0
}
}{\mathop{\sum}}\frac{\ell!}{\left(2r\right)!j!}\right.\\
 & \left.\times\left(\frac{k\pi}{2}\right)^{r}a_{j}\left(\tau_{3}^{*}\right)\left(-z\right)^{-r-j}\right].
\end{aligned}
\]
It is easy to see that $\left|a_{j}\left(\tau_{3}^{*}\right)\right|\ll1$
as $z\rightarrow0$ with $\textrm{Re}\left(z\right)>0.$ Thus, by
Lemma \ref{l2.2} and $\textrm{Re}\left(\frac{1}{z}\right)\geq\frac{k}{2},$
\begin{equation}
\left[\frac{\partial^{\ell}}{\partial u^{\ell}}e^{-\frac{3\pi ku^{2}}{2z}}\textrm{OU}_{1}\left(\frac{2u}{\textrm{i}z};2\tau_{3}^{*}\right)\right]_{u=0}\ll\left|z\right|^{-\ell}e^{-\frac{\pi}{k}\textrm{Re}\left(\frac{1}{z}\right)}.\label{eq:3.11}
\end{equation}
Consequently, the main term for odd $k$ comes from the second term
in \eqref{eq:3.2}. Proceeding as in \eqref{eq:3.8}, we can obtain
that
\[
\begin{aligned}\frac{C^{*}\left(u;\tau\right)}{\eta\left(\tau\right)}\frac{\eta\left(2\tau\right)}{C^{*}\left(u;2\tau\right)} & =e^{-\frac{\pi\textrm{i}}{6k}\left(\left[-h\right]_{k}^{*}-2\left[-h\right]_{k/2}^{*}\right)-\frac{\pi}{6kz}}\frac{C^{*}\left(\frac{u}{\textrm{i}z};\tau_{1}^{*}\right)}{\eta\left(\tau_{1}^{*}\right)}\frac{\eta\left(2\tau_{3}^{*}\right)}{C^{*}\left(\frac{u}{\textrm{i}z};2\tau_{3}^{*}\right)}\\
 & =e^{\frac{\pi\textrm{i}}{6k}\left(\left[-h\right]_{k}^{*}-2\left[-h\right]_{k/2}^{*}\right)-\frac{\pi}{6kz}}\left(1+\sum_{j\geq0}\gamma_{j}\left(\tau_{1}^{*},\tau_{3}^{*}\right)\frac{\left(2\pi\textrm{i}u\right)^{j}}{j!}\right),
\end{aligned}
\]
where $\left|\gamma_{j}\left(\tau_{1}^{*},\tau_{3}^{*}\right)\right|\ll e^{2\pi\textrm{i}\tau}$
as $\tau\rightarrow\textrm{i}\infty.$ Therefore, 
\[
\begin{aligned} & -\frac{1}{2}\chi_{h,k/2}\left(\textrm{i}z\right)^{-1/2}\frac{C^{*}\left(\frac{u}{\textrm{i}z};\tau_{1}^{*}\right)}{\eta\left(\tau_{1}^{*}\right)}\frac{\eta\left(2\tau_{3}^{*}\right)}{C^{*}\left(\frac{u}{\textrm{i}z};2\tau_{3}^{*}\right)}e^{\frac{\pi ku^{2}}{2z}}\mathcal{E}_{\frac{2\left[-h\right]_{k/2}^{*}}{k}}\left(\frac{2u}{\textrm{i}z};2\tau_{3}^{*}\right)\\
 & =-\frac{1}{2}\chi_{h,k/2}\left(\textrm{i}z\right)^{-1/2}e^{\frac{\pi ku^{2}}{2z}}\left[e^{\frac{\pi\textrm{i}}{6k}\left(\left[-h\right]_{k}^{*}-2\left[-h\right]_{k/2}^{*}\right)}\mathcal{E}_{\frac{2\left[-h\right]_{k/2}^{*}}{k},\frac{1}{24}}^{*}\left(\frac{2u}{\textrm{i}z};2\tau_{3}^{*}\right)\right.\\
 & +e^{\frac{\pi\textrm{i}}{6k}\left(\left[-h\right]_{k}^{*}-2\left[-h\right]_{k/2}^{*}\right)}\mathcal{E}_{\frac{2\left[-h\right]_{k/2}^{*}}{k},\frac{1}{24}}^{e}\left(\frac{2u}{\textrm{i}z};2\tau_{3}^{*}\right)\\
 & \left.+\left(\frac{C^{*}\left(\frac{u}{\textrm{i}z};\tau_{1}^{*}\right)}{\eta\left(\tau_{1}^{*}\right)}\frac{\eta\left(2\tau_{3}^{*}\right)}{C^{*}\left(\frac{u}{\textrm{i}z};2\tau_{3}^{*}\right)}-e^{\frac{\pi\textrm{i}}{6k}\left(\left[-h\right]_{k}^{*}-2\left[-h\right]_{k/2}^{*}\right)}\right)\mathcal{E}_{\frac{2\left[-h\right]_{k/2}^{*}}{k},0}^{e}\left(\frac{2u}{\textrm{i}z};2\tau_{3}^{*}\right)\right].
\end{aligned}
\]
Applying Lemma \ref{l3.3}, we have
\begin{equation}
\left[\frac{\partial^{\ell}}{\partial u^{\ell}}e^{\frac{\pi ku^{2}}{2z}}\mathcal{E}_{\frac{2\left[-h\right]_{k/2}^{*}}{k},D_{2}}^{e}\left(\frac{u}{\textrm{i}z};2\tau_{3}^{*}\right)\right]_{u=0}\ll\log\left(k\right)\left|z\right|^{-\ell},\label{eq:3.12}
\end{equation}
as $0\leq D_{2}\leq\frac{1}{24}.$ By \eqref{eq:3.11} and \eqref{eq:3.12},
we can obtain that
\begin{equation}
\begin{alignedat}{1}\left[\frac{\partial^{\ell}}{\partial u^{\ell}}\textrm{OU}_{1}\left(u;\tau\right)\right]_{u=0}= & -\frac{1}{2}\chi_{h,k/2}\left(\textrm{i}z\right)^{-1/2}e^{\frac{\pi\textrm{i}}{6k}\left(\left[-h\right]_{k}^{*}-2\left[-h\right]_{k/2}^{*}\right)}\\
 & \times\left[\frac{\partial^{\ell}}{\partial u^{\ell}}\left(e^{\frac{\pi ku^{2}}{2z}}\mathcal{E}_{\frac{2\left[-h\right]_{k/2}^{*}}{k},\frac{1}{24}}^{*}\left(\frac{2u}{\textrm{i}z};2\tau_{3}^{*}\right)\right)\right]_{u=0}\\
 & +O\left(\log\left(k\right)\left|z\right|^{-\frac{1}{2}-\ell}\right).
\end{alignedat}
\label{eq:3.13}
\end{equation}
Using Lemma \ref{l3.4}, we have
\[
\begin{aligned} & \left[\frac{\partial^{\ell}}{\partial u^{\ell}}\left(e^{\frac{\pi ku^{2}}{2z}}\mathcal{E}_{\frac{2\left[-h\right]_{k/2}^{*}}{k},\frac{1}{24}}^{*}\left(\frac{2u}{\textrm{i}z};2\tau_{3}^{*}\right)\right)\right]_{u=0}=-\frac{\textrm{i}}{2\sqrt{6}\pi}\sum_{v=0}^{k-1}\left(-1\right)^{v}e^{\pi\textrm{i}\left(v+\frac{1}{2}\right)^{2}\frac{2\left[-h\right]_{k/2}^{*}}{k}}\\
 & \times\int_{-1}^{1}\left[\frac{\partial^{\ell}}{\partial u^{\ell}}\left(e^{\frac{\pi ku^{2}}{2z}}\cot\left(\frac{\pi}{k}\left(\frac{x}{\sqrt{6}}-v-\frac{1}{2}+ku\right)\right)\right)\right]_{u=0}e^{-\frac{\pi}{6kz}\left(1+x^{2}\right)}dx.
\end{aligned}
\]
Then
\[
\begin{aligned} & \left[\frac{\partial^{\ell}}{\partial u^{\ell}}\left(e^{\frac{\pi ku^{2}}{2z}}\cot\left(\frac{\pi}{k}\left(\frac{x}{\sqrt{6}}-v-\frac{1}{2}+ku\right)\right)\right)\right]_{u=0}\\
 & =\left(2\pi\textrm{i}\right)^{\ell}\underset{\substack{0\leq j\leq\ell/2}
}{\mathop{\sum}}\left(\begin{array}{c}
\ell\\
2j
\end{array}\right)\left(\frac{\pi k}{2}\right)^{j}z^{-j}C_{\ell-2j}\left(\frac{1}{2k}\left(\frac{x}{2\sqrt{3}}-v-\frac{1}{2}\right)\right).
\end{aligned}
\]
Substituting this into \eqref{eq:3.13} completes the proof of \eqref{eq:3.6}.
\end{proof}

\subsection{Proof of Theorem \ref{t1.1}}

By Lemma \ref{l3.1}, we have
\[
\textrm{ou}_{\ell}\left(n\right)=\textrm{coeff}_{\left[q^{n}\right]}\frac{1}{\left(2\pi\textrm{i}\right)^{\ell}}\left[\frac{\partial^{\ell}}{\partial u^{\ell}}\left[q^{-\frac{1}{4}}\textrm{OU}_{1}\left(u;\tau\right)+H_{1}\left(\zeta;q\right)\right]\right]_{u=0}.
\]
It is easy to show that
\[
\begin{aligned}\frac{1}{\left(2\pi\textrm{i}\right)^{\ell}}\left[\frac{\partial^{\ell}}{\partial u^{\ell}}H_{1}\left(\zeta;q\right)\right]_{u=0} & =\sum_{n\geq0}\left(-1\right)^{n+1}q^{3n^{2}+2n}\left(\left(3n+1\right)^{\ell}+\left(3n+2\right)^{\ell}q^{2n+1}\right)\\
 & =q^{-\frac{1}{3}}\sum_{j\geq0}\left[\left(6\left(j+\frac{2}{3}\right)\right)^{\ell}q^{12\left(j+\frac{2}{3}\right)^{2}}+\left(6\left(j+\frac{5}{6}\right)\right)^{\ell}q^{12\left(j+\frac{5}{6}\right)^{2}}\right.\\
 & \left.-\left(6\left(j+\frac{1}{6}\right)\right)^{\ell}q^{12\left(j+\frac{1}{6}\right)^{2}}-\left(6\left(j+\frac{1}{3}\right)\right)^{\ell}q^{12\left(j+\frac{1}{3}\right)^{2}}\right]\\
 & =6^{\ell}z^{-\frac{\ell}{2}}q^{-\frac{1}{3}}\sum_{j\geq0}\left[f_{\ell}\left(\sqrt{z}\left(j+\frac{2}{3}\right)\right)+f_{\ell}\left(\sqrt{z}\left(j+\frac{5}{6}\right)\right)\right.\\
 & \left.-f_{\ell}\left(\sqrt{z}\left(j+\frac{1}{6}\right)\right)-f_{\ell}\left(\sqrt{z}\left(j+\frac{1}{3}\right)\right)\right],
\end{aligned}
\]
where $f_{\ell}\left(z\right)\coloneqq z^{\ell}e^{-12z^{2}}.$ Since
$q^{-\frac{1}{3}}=O\left(1\right),$ we use Proposition \ref{p2.1}
with $a\in\left\{ \frac{1}{6},\frac{1}{3},\frac{2}{3},\frac{5}{6}\right\} $
to obtain
\[
\left|\frac{1}{\left(2\pi\textrm{i}\right)^{\ell}}\left[\frac{\partial^{\ell}}{\partial u^{\ell}}H_{1}\left(\zeta;q\right)\right]_{u=0}\right|\ll\left|z\right|^{-\frac{\ell}{2}}\ll\frac{n^{\frac{\ell}{2}}}{k^{\frac{\ell}{2}}}\ll n^{\frac{\ell}{2}}.
\]
Thus this part contributes to the error term.

Define
\[
\textrm{OU}_{\ell}\left(\tau\right)\coloneqq\frac{1}{\left(2\pi\textrm{i}\right)^{\ell}}\left[\frac{\partial^{\ell}}{\partial u^{\ell}}\textrm{OU}_{1}\left(u;\tau\right)\right]_{u=0}\eqqcolon\sum_{n\geq0}a_{\ell}\left(n\right)q^{n+\frac{1}{4}}.
\]
With two coprime integers $h$ and $k,$ we define three consecutive
fractions in the Farey sequence of order $N$ by $\frac{h_{1}}{k_{1}}<\frac{h}{k}<\frac{h_{2}}{k_{2}},$
where $N=\left\lfloor \sqrt{n}\right\rfloor .$ Let $z=\frac{k}{n}-\textrm{i}k\varTheta,$
$\vartheta_{h,k}^{\prime}=\frac{1}{k\left(k_{1}+k\right)}$ and $\vartheta_{h,k}^{\prime\prime}=\frac{1}{k\left(k_{2}+k\right)}.$
Then we have
\[
a_{\ell}\left(n\right)=\mathop{\underset{\substack{0\le h<k\le N\\
\gcd\left(h,k\right)=1
}
}{\mathop{\sum}}e^{-\frac{2\pi\textrm{i}}{k}\left(n+\frac{1}{4}\right)h}\int_{-\vartheta_{h,k}^{\prime}}^{\vartheta_{h,k}^{\prime\prime}}\textrm{OU}_{\ell}\left(\tau\right)e^{\frac{2\pi\left(n+\frac{1}{4}\right)z}{k}}d\varTheta.}
\]

We now apply Theorem \ref{t3.2}. The contribution of the error term
may be bounded against
\[
\begin{aligned} & \underset{\substack{0\le h<k\le N\\
\gcd\left(h,k\right)=1
}
}{\mathop{\sum}}\log\left(k\right)\int_{-\vartheta_{h,k}^{\prime}}^{\vartheta_{h,k}^{\prime\prime}}\left|z\right|^{-\frac{1}{2}-\ell}e^{\frac{2\pi\left(n+\frac{1}{4}\right)z}{k}}d\varTheta\\
 & \ll\sum_{1\leq k\leq N}k\log\left(k\right)\int_{-\vartheta_{h,k}^{\prime}}^{\vartheta_{h,k}^{\prime\prime}}\left|z\right|^{-\frac{1}{2}-\ell}e^{2\pi\left(1+\frac{1}{4n}\right)}d\varTheta.
\end{aligned}
\]
Since $\vartheta_{h,k}^{\prime},\vartheta_{h,k}^{\prime\prime}\asymp\frac{1}{kN},$
we can bound the error term as
\[
\ll\frac{n^{\ell+\frac{1}{2}}}{N}\sum_{1\leq k\leq N}\frac{\log\left(k\right)}{k^{\ell+\frac{1}{2}}}\ll\frac{n^{\ell+\frac{1}{2}}}{N}\ll n^{\ell+\frac{3}{4}}.
\]
The main term for even $k$ becomes 
\[
\begin{aligned}U_{1}\coloneqq & \frac{1}{4\sqrt{6}}\underset{\substack{0\leq j\leq\ell/2}
}{\mathop{\sum}}\left(\begin{array}{c}
\ell\\
2j
\end{array}\right)\frac{\pi^{j}}{j!2^{j}}\underset{\substack{0\le h<k\le N\\
\gcd\left(h,k\right)=1\\
\gcd\left(k,2\right)=2\\
0\leq v\leq k-1
}
}{\mathop{\sum}}k^{j-1}K_{k,2}\left(n,v\right)\\
 & \times\int_{-1}^{1}C_{\ell-2j}\left(\frac{1}{k}\left(\frac{x}{\sqrt{6}}-v-\frac{1}{2}\right)\right)\int_{-\vartheta_{h,k}^{\prime}}^{\vartheta_{h,k}^{\prime\prime}}z^{-\frac{1}{2}-j}e^{-\frac{\pi}{6kz}\left(1+x^{^{2}}\right)+\frac{2\pi\left(n+\frac{1}{4}\right)z}{k}}d\varTheta dx,
\end{aligned}
\]
which, however, contributes to the error term. The function $U_{1}$
can be estimated against
\[
\begin{aligned}\left|U_{1}\right| & \ll\underset{\substack{0\leq j\leq\ell/2}
}{\mathop{\sum}}\underset{\substack{1\le k\le N\\
\gcd\left(k,2\right)=2\\
0\leq v\leq k-1
}
}{\mathop{\sum}}k^{j-1}\int_{-1}^{1}C_{\ell-2j}\left(\frac{1}{k}\left(\frac{x}{\sqrt{6}}-v-\frac{1}{2}\right)\right)\\
 & \times\int_{-\vartheta_{h,k}^{\prime}}^{\vartheta_{h,k}^{\prime\prime}}\left|z\right|^{-\frac{1}{2}-j}e^{2\pi\left(1+\frac{1}{4n}\right)}d\varTheta dx\\
 & \ll\underset{\substack{1\le k\le N\\
\gcd\left(k,2\right)=2\\
0\leq v\leq k-1
}
}{\mathop{\sum}}\frac{k^{\frac{\ell}{2}-2}}{N}\frac{n^{\frac{\ell}{2}+\frac{1}{2}}}{k^{\frac{\ell}{2}+\frac{1}{2}}}\int_{-1}^{1}\cot\left(\frac{\pi}{k}\left(\frac{x}{\sqrt{6}}-v-\frac{1}{2}\right)\right)dx\\
 & \ll\underset{\substack{1\le k\le N\\
\gcd\left(k,2\right)=2
}
}{\mathop{\sum}}\frac{k^{\frac{\ell}{2}-2}}{N}\frac{n^{\frac{\ell}{2}+\frac{1}{2}}}{k^{\frac{\ell}{2}+\frac{1}{2}}}k\ll n^{\frac{\ell+1}{2}}\ll n^{\ell+\frac{3}{4}}.
\end{aligned}
\]
Consequently, the main term becomes 
\[
\begin{aligned} & \frac{1}{16\sqrt{6}}\underset{\substack{1\le k\le N\\
\gcd\left(k,2\right)=1\\
0\leq v\leq k-1
}
}{\mathop{\sum}}\frac{K_{k,1}\left(n,v\right)}{k}\underset{\substack{0\leq j\leq\ell/2}
}{\mathop{\sum}}\left(\begin{array}{c}
\ell\\
2j
\end{array}\right)\left(-\frac{1}{4}\right)^{j}\\
 & \times\underset{\substack{a+b=j\\
a,b\geq0
}
}{\mathop{\sum}}k^{a}\kappa\left(a,b\right)\int_{-1}^{1}C_{\ell-2j}\left(\frac{1}{2k}\left(\frac{x}{2\sqrt{3}}-v-\frac{1}{2}\right)\right)\\
 & \times\int_{-\vartheta_{h,k}^{\prime}}^{\vartheta_{h,k}^{\prime\prime}}z^{-\frac{1}{2}-a-2b}e^{\frac{\pi}{12kz}\left(1-x^{^{2}}\right)+\frac{2\pi\left(n+\frac{1}{4}\right)z}{k}}d\varTheta dx.
\end{aligned}
\]
Applying Lemma \ref{l2.3} with $v=a+2b+\frac{1}{2},$ $A=\frac{2\pi\left(n+\frac{1}{4}\right)}{k},$
$B=\frac{\pi}{12k}\left(1-x^{^{2}}\right),$ $\vartheta_{1}=\vartheta_{h,k}^{\prime\prime}$
and $\vartheta_{2}=\vartheta_{h,k}^{\prime},$ we complete the proof.
\qed

\subsection{Proof of Corollary \ref{c1.1}}

Following the approach of \cite[Theorem 1.4]{BB} and \cite[Theorem 1.1(1)]{BJM},
by taking $k=1,$ $j=\frac{\ell}{2},$ $a=0$ and $b=j,$ we conclude
the proof of Corollary \ref{c1.1}. \qed

\subsection{Proof of Theorem \ref{t1.2}}

 By Corollary \ref{c1.1} and \eqref{eq:1.5}, we have
\[
\frac{\textrm{ou}_{2\ell}\left(n\right)}{\textrm{ou}\left(n\right)}\sim\left(-6n\right)^{\ell}E_{2\ell}\left(\frac{1}{2}\right),
\]
as $n\rightarrow\infty.$ From \cite[(23.1)]{AS}, we have $E_{2\ell}\left(\frac{1}{2}\right)=2^{-2\ell}E_{2\ell},$
$\left(-1\right)^{\ell}E_{2\ell}>0$ and $E_{2\ell+1}=0.$ Then we
can obtain that

\[
\frac{\textrm{ou}_{\ell}\left(n\right)}{\textrm{ou}\left(n\right)\left(\frac{3n}{2}\right)^{\ell/2}}\sim\left|E_{\ell}\right|.
\]
It is well-established that $\left|E_{\ell}\right|$ corresponds to
the $\ell$-th moment of the hyperbolic secant distribution with mean
0 and scale 1 \cite[pp.147-148]{JKB}. An application of the Method
of Moments completes the proof. \qed

\section{\label{sec:4}Proof of Theorems \ref{t1.3}, \ref{t1.4}, \ref{t1.5}
and \ref{t1.6}}

In this section, we prove Theorems \ref{t1.3}, \ref{t1.4}, \ref{t1.5}
and \ref{t1.6}. First, we introduce the conditioned Boltzmann model
for the odd unimodel sequences that relates the measure $\boldsymbol{\textrm{P}}_{n}$
to the limiting distributions established in our main theorems.

\subsection{Conditioned Boltzmann model}

Before introuducing the conditioned Boltzmann model, we recall the
Boltzmann model in \cite{FB}. Let $\textrm{P}_{n}$ denote the uniform
probability measure on partitions of size $\ensuremath{n}.$ Define
the random variables $X_{k}$ as the number of parts of size $k$
in a partition, and let $N=\sum_{k\geq1}kX_{k}$ denote the partition
size. The central observation is that while the $X_{k}$ are not independent
under $\ensuremath{P_{n}},$ they become independent under the measure
$\textrm{\textbf{Q}}_{q}.$ The probability measure $\textrm{\textbf{Q}}_{q}$
is defined for any $q\in\left(0,1\right)$ on the set of all odd unimodal
sequences $\mathcal{OU}$ by
\[
\textrm{\textbf{Q}}_{q}\left(\lambda\right)\coloneqq\frac{q^{\left|\lambda\right|}}{\textrm{OU}\left(q\right)},
\]
where the normalizing function is given by 
\[
\textrm{OU}\left(q\right)\coloneqq\sum_{m\geq0}\frac{q^{2m+1}}{\left(q;q^{2}\right)_{m+1}^{2}}
\]
as defined in equation \eqref{eq:1.3}. The probability measure $\textrm{\textbf{Q}}_{q}$
is not directly useful, since there is not a simple expression for
the individual probabilities $\textrm{\textbf{Q}}_{q}\left(X_{2k-1}^{\left[L\right]}=\ell\right)$
and $\textrm{\textbf{Q}}_{q}\left(X_{2k-1}^{\left[R\right]}=\ell\right)$,
and the sequences $\left\{ X_{2k-1}^{\left[L\right]}\right\} $ and
$\left\{ X_{2k-1}^{\left[R\right]}\right\} $ are not independent,
which implies that $\textrm{OU}\left(q\right)$ is not a product.
However, the $\left(m+1\right)$-th summand in $\textrm{OU}\left(q\right)$
is of course the product
\[
q^{2m+1}\prod_{k=1}^{m+1}\frac{1}{\left(1-q^{2k-1}\right)^{2}}=\underset{\substack{\lambda\\
\textrm{PK}\left(\lambda\right)=2m+1
}
}{\mathop{\sum}}q^{\left|\lambda\right|}.
\]
By conditioning $\textrm{\textbf{Q}}_{q}$ on the event $\textrm{PK}=2m+1,$
we can obtain tractable expressions for the individual probabilities
of $X_{2k-1}^{\left[L\right]},$ thereby enabling the full use of
Fristedt\textquoteright s techniques from \cite{FB}. Crucially, this
procedure can be carried out uniformly over the contributing range
of $m,$ allowing us to assemble the local distributions into the
desired global results. We will conclude Proposition \ref{p4.3} in
this subsection, a direct analogue of \cite[Proposition 4.6]{FB}
and \cite[Proposition 4.6]{BB2}.

Set $\textrm{\textbf{Q}}_{q,m}\coloneqq\textrm{\textbf{Q}}_{q}\left(\cdot\mid\textrm{PK}=2m+1\right)$
and $\textrm{\textbf{P}}_{n,m}\coloneqq\textrm{\textbf{P}}_{n}\left(\cdot\mid\textrm{PK}=2m+1\right).$
Let $\mathcal{OU}_{n,m}\subset\mathcal{OU}\left(n\right)$ be those
sequences with peak $2m+1,$ and $ou_{m}\left(n\right)\coloneqq\#\mathcal{OU}_{n,m}.$
The following lemma follows directly from the definitions, which is
an analogue of \cite[Lemmas 4.1 and 4.2]{BB2}.
\begin{lem}
\label{l4.1}(i) We have
\[
\textrm{\textbf{Q}}_{q,m}\left(\lambda\right)=\begin{cases}
\left(q;q^{2}\right)_{m+1}^{2}q^{\left|\lambda\right|-\left(2m+1\right)} & \textrm{if PK}\left(\lambda\right)=2m+1,\\
0 & \textrm{otherwise.}
\end{cases}
\]

(ii) The set $\left\{ X_{2k-1}^{\left[j\right]}\right\} _{k\geq1,j\in\left\{ L,R\right\} }$
is a set of independent random variables under $\textrm{\textbf{Q}}_{q,m}$
with probability densities
\[
\textrm{\textbf{Q}}_{q,m}\left(X_{2k-1}^{\left[j\right]}=\ell\right)=\begin{cases}
\left(1-q^{2k-1}\right)q^{\ell\left(2k-1\right)} & \textrm{if }k\leq m+1,\\
0 & \textrm{otherwise.}
\end{cases}
\]
In particular,
\[
\begin{aligned}\textrm{\textbf{Q}}_{q,m}\left(N=n\right) & =\left[\zeta^{n-\left(2m+1\right)}\right]\frac{\left(q;q^{2}\right)_{m+1}^{2}}{\left(\zeta q,;\left(\zeta q\right)^{2}\right)_{m+1}}\\
 & =\left(q;q^{2}\right)_{m+1}^{2}q^{n-\left(2m+1\right)}ou_{m}\left(n\right).
\end{aligned}
\]

(iii) We have
\[
\textrm{\textbf{P}}_{n,m}=\boldsymbol{Q}_{q,m}\left(\cdot\mid N=n\right).
\]
\end{lem}
Intuitively, the parameter $q=q\left(n\right)\in\left(0,1\right)$
is chosen to maximize 
\[
\textrm{\textbf{Q}}_{q,m}\left(N=n\right)=ou_{m}\left(n\right)q^{n-\left(2m+1\right)}\left(q;q^{2}\right)_{m+1}^{2},
\]
following the saddle-point principle applied to $\ensuremath{q^{\left(2m+1\right)-n}\left(q;q^{2}\right)_{m+1}^{-2}}.$
Similarly to \cite{BW,BB2,FB,RP}, we set $q=q\left(n\right)=e^{-\frac{1}{B\sqrt{n}}},$
which is independent of $m$ and $B=\frac{\sqrt{6}}{\pi}.$ We fix
this choice of $q$ throughout our analysis of odd unimodal sequences.
The parameter $m$ is expressed in terms of $r\in\mathbb{R}$ as
\[
2m+1=B\sqrt{n}\left(r+\log\left(2B\sqrt{n}\right)\right).
\]
To guarantee $m\in\mathbb{Z},$ we restrict $r$ to $\frac{1}{B\sqrt{n}}\left(2\mathbb{Z}+1-\log\left(2B\sqrt{n}\right)\right).$
For any $\left[r_{1},r_{2}\right]\subset\mathbb{R},$ we can divide
$\textrm{\textbf{P}}_{n}$ into the ranges
\[
\textrm{\textbf{P}}_{n}=\left(\sum_{r<r_{1}}+\sum_{r\in\left[r_{1},r_{2}\right]}+\sum_{r>r_{2}}\right)\textrm{\textbf{P}}_{n}\left(\textrm{PK}=2m+1\right)\cdot\textrm{\textbf{P}}_{n,m},
\]
and we can bound the tail ranges for any measurable set $S$ as
\begin{equation}
\begin{aligned}\sum_{r<r_{1}}\textrm{\textbf{P}}_{n}\left(\textrm{PK}=2m+1\right)\cdot\textrm{\textbf{P}}_{n,m}\left(S\right) & \leq\sum_{r<r_{1}}\textrm{\textbf{P}}_{n}\left(\textrm{PK}=2m+1\right)\\
 & =\textrm{\textbf{P}}_{n}\left(\frac{\textrm{PK}-B\sqrt{n}\log\left(2B\sqrt{n}\right)}{B\sqrt{n}}<r_{1}\right),
\end{aligned}
\label{eq:4.1}
\end{equation}
and
\begin{equation}
\sum_{r>r_{2}}\textrm{\textbf{P}}_{n}\left(\textrm{PK}=2m+1\right)\cdot\textrm{\textbf{P}}_{n,m}\left(S\right)\leq\textrm{\textbf{P}}_{n}\left(\frac{\textrm{PK}-B\sqrt{n}\log\left(2B\sqrt{n}\right)}{B\sqrt{n}}>r_{2}\right).\label{eq:4.2}
\end{equation}
For sequences $a_{n}\leq b_{n}$ of positive integers, we define the
random vector 
\[
\boldsymbol{X}_{\left[a_{n},b_{n}\right]}:=\left(X_{2k-1}^{\left[j\right]}\right)_{k\in\left[a_{n},b_{n}\right],\ j\in\left\{ L,R\right\} }.
\]
To establish our main theorems, we proceed as follows:

(i) Show that \eqref{eq:4.1} and \eqref{eq:4.2} tend to 0 as $r_{1}\to-\infty$
and $\ensuremath{r_{2}\to\infty},$ respectively.

(ii) Prove that under a mild condition on $a_{n}$ and $\ensuremath{b_{n}},$
\[
d_{\text{TV}}\left(\textrm{\textbf{P}}_{n,m}\left(\boldsymbol{X}_{\left[a_{n},b_{n}\right]}^{-1}\right),\ \textrm{\textbf{Q}}_{q,m}\left(\boldsymbol{X}_{\left[a_{n},b_{n}\right]}^{-1}\right)\right)\to0
\]
uniformly for $\ensuremath{r\in\left[r_{1},r_{2}\right]}.$

To assist the proof of (ii), we state the following analogue of \cite[Lemma 4.2]{FB},
which follows by the same argument.
\begin{prop}
\label{p4.1}\cite[Proposition 4.3]{BB2}Let $a_{n}\leq b_{n}$ be
sequences of integers and suppose that there exist sets $B_{n,m}\subset\mathbb{R}^{2\left(b_{n}-a_{n}+1\right)}$
such that, uniformly for $r$ in any compact interval $\ensuremath{\left[r_{1},r_{2}\right]},$

(i) $\ensuremath{\textrm{\textbf{Q}}_{n,m}\bigl(\boldsymbol{X}_{\left[a_{n},b_{n}\right]}\in B_{n,m}\bigr)\to1},$

(ii) $\frac{\textbf{Q}_{q,m}\bigl(N=n\mid\boldsymbol{X}_{\left[a_{n},b_{n}\right]}=\boldsymbol{x}\bigr)}{\textbf{Q}_{q,m}\left(N=n\right)}\to1$
\textup{uniformly in $x\in B_{n,m}.$}

Then $d_{\text{TV}}\left(\textrm{\textbf{P}}_{n,m}\left(\boldsymbol{X}_{\left[a_{n},b_{n}\right]}^{-1}\right),\ \textrm{\textbf{Q}}_{q,m}\left(\boldsymbol{X}_{\left[a_{n},b_{n}\right]}^{-1}\right)\right)\to0$
uniformly for $\ensuremath{r\in\left[r_{1},r_{2}\right]}.$
\end{prop}
If the sequences $a_{n}$ and $b_{n}$ satisfy a simple condition,
then the sets $B_{n,m}$ required in Proposition \,\ref{p4.1} indeed
exist; this fact will be established in Proposition \,\ref{p4.3}.
Before doing so, we first derive the asymptotic behavior of the denominator
appearing in condition\,(ii) of Proposition \,\ref{p4.1}.
\begin{prop}
\label{p4.2}Suppose that $q=e^{-\frac{1}{B\sqrt{n}}}.$ Uniformly
for $r$ in any $\left[r_{1},r_{2}\right],$ we have
\begin{equation}
\textrm{\textbf{Q}}_{q,m}\left(N=n\right)\sim\frac{1}{2^{\frac{5}{4}}\cdot3^{\frac{1}{4}}n^{\frac{3}{4}}},\label{eq:4.3}
\end{equation}
\begin{equation}
\textrm{\textbf{P}}_{n}\left(\textrm{PK}=2m+1\right)\sim\frac{1}{B\sqrt{n}}e^{-r-\frac{1}{2}e^{-r}}.\label{eq:4.4}
\end{equation}
\end{prop}
\begin{proof}
By Lemma \ref{l4.1}, we have
\begin{equation}
\textrm{\textbf{Q}}_{q,m}\left(N=n\right)=\left(q;q^{2}\right)_{m+1}^{2}q^{n-\left(2m+1\right)}ou_{m}\left(n\right).\label{eq:4.5}
\end{equation}
To estimate the right-hand side, we first write $ou_{m}\left(n\right)$
as a Cauchy integral and then apply the saddle-point method (Proposition\,\ref{p2.3}).
Although this analytic approach is essentially equivalent to the probabilistic
one used in \cite[Propoistion 4.5]{FB} and \cite[Propoistion 3]{RP},
it allows us to omit details that closely parallel calculations already
given in \cite{BW,RP}. Similarly to \cite{BB2}, we first rewrite
\[
ou_{m}\left(n\right)=\frac{1}{2\pi\textrm{i}}\int_{\mathcal{C}}\frac{\zeta^{\left(2m+1\right)-n-1}}{\left(\zeta;\zeta^{2}\right)_{m+1}^{2}}d\zeta,
\]
where $\mathcal{C}$ is a circle centered at 0 with radius less than
1 oriented counterclockwise. Substituting $\zeta=e^{-\frac{1}{B\sqrt{n}}+2\pi\textrm{i}\theta},$
we have
\[
\begin{aligned}ou_{m}\left(n\right) & =\frac{1}{2\pi\textrm{i}}\int_{-\frac{1}{2}}^{\frac{1}{2}}\frac{e^{-\frac{1}{B\sqrt{n}}\left[\left(2m+1\right)-n\right]+2\pi\textrm{i}\left[\left(2m+1\right)-n\right]\theta}}{\left(e^{-\frac{1}{B\sqrt{n}}+2\pi\textrm{i}\theta};e^{2\left(-\frac{1}{B\sqrt{n}}+2\pi\textrm{i}\theta\right)}\right)_{m+1}^{2}}d\theta\\
 & =\frac{1}{2\pi\textrm{i}}\int_{-\frac{1}{2}}^{\frac{1}{2}}\exp\left(f\left(2\pi\textrm{i}\theta\right)\right)d\theta,
\end{aligned}
\]
where 
\[
f\left(z\right)\coloneqq\frac{n-\left(2m+1\right)}{B\sqrt{n}}+\left[\left(2m+1\right)-n\right]z-2\sum_{k=1}^{m+1}\textrm{Log}\left(1-e^{-\frac{2k-1}{B\sqrt{n}}+\left(2k-1\right)z}\right).
\]
Throughout the proof, we suppress the dependence of $f$ on $n$ for
notational simplicity. To verify condition (i) in Proposition\,\ref{p2.3},
it suffices to establish the asymptotic behavior of $f\left(0\right),$
$f^{\prime}\left(0\right)$ and $f^{\prime\prime}\left(0\right).$
These derivatives are straightforward to compute:
\[
f\left(0\right)=\frac{n-\left(2m+1\right)}{B\sqrt{n}}-2\sum_{k=1}^{m+1}\log\left(1-e^{-\frac{2k-1}{B\sqrt{n}}}\right),
\]
\[
f^{\prime}\left(0\right)=\left(2m+1\right)-n+2\sum_{k=1}^{m+1}\frac{2k-1}{e^{\frac{2k-1}{B\sqrt{n}}}-1},
\]
\[
f^{\prime\prime}\left(0\right)=2\sum_{k=1}^{m+1}\frac{\left(2k-1\right)^{2}e^{-\frac{2k-1}{B\sqrt{n}}}}{\left(1-e^{-\frac{2k-1}{B\sqrt{n}}}\right)^{2}}.
\]
Applying Lemma \ref{l2.7} with $x\mapsto e^{-\frac{2k-1}{B\sqrt{n}}}$
and $s\mapsto2\pi\left(2k-1\right)\theta,$ we can obtain that
\[
\begin{aligned} & \left|f\left(2\pi\textrm{i}\theta\right)-f\left(0\right)-f^{\prime}\left(0\right)2\pi\textrm{i}\theta-f^{\prime\prime}\left(0\right)\frac{\left(2\pi\textrm{i}\theta\right)^{2}}{2}\right|\\
 & =\left|\frac{n-\left(2m+1\right)}{B\sqrt{n}}+2\pi\textrm{i}\left[\left(2m+1\right)-n\right]\theta-2\sum_{k=1}^{m+1}\log\left(1-e^{-\frac{2k-1}{B\sqrt{n}}+\left(2k-1\right)\theta}\right)\right.\\
 & -\frac{n-\left(2m+1\right)}{B\sqrt{n}}-2\sum_{k=1}^{m+1}\log\left(1-e^{-\frac{2k-1}{B\sqrt{n}}}\right)\\
 & \left.-\left[\left(2m+1\right)-n+2\sum_{k=1}^{m+1}\frac{2k-1}{e^{\frac{2k-1}{B\sqrt{n}}}-1}\right]2\pi\textrm{i}\theta-\sum_{k=1}^{m+1}\frac{\left(2k-1\right)^{2}e^{-\frac{2k-1}{B\sqrt{n}}}}{\left(1-e^{-\frac{2k-1}{B\sqrt{n}}}\right)^{2}}\left(2\pi\textrm{i}\theta\right)^{2}\right|\\
 & =2\left|\sum_{k=1}^{m+1}\left(\log\left(\frac{1-e^{-\frac{2k-1}{B\sqrt{n}}}}{1-e^{-\frac{2k-1}{B\sqrt{n}}}+2\pi i\left(2k-1\right)\theta}\right)-\frac{e^{-\frac{2k-1}{B\sqrt{n}}}}{1-e^{-\frac{2k-1}{B\sqrt{n}}}}2\pi\textrm{i}\left(2k-1\right)\theta\right.\right..\\
 & \left.\left.+\frac{e^{-\frac{2k-1}{B\sqrt{n}}}}{\left(1-e^{-\frac{2k-1}{B\sqrt{n}}}\right)^{2}}\frac{\left(2\pi\textrm{i}\left(2k-1\right)\theta\right)^{2}}{2}\right)^{2}\right|\\
 & \leq C\sum_{k=1}^{m+1}\frac{\left(2k-1\right)^{3}e^{-\frac{2k-1}{B\sqrt{n}}}}{\left(1-e^{-\frac{2k-1}{B\sqrt{n}}}\right)^{3}}\left|\theta\right|^{3}.
\end{aligned}
\]
for some constant $C.$ Recognizing Riemann sums for a convergent
integral, the sum is bounded by
\[
\begin{aligned} & B^{4}n^{2}\sum_{k\geq1}\frac{\left(\frac{2k-1}{B\sqrt{n}}\right)^{3}e^{-\frac{2k-1}{B\sqrt{n}}}}{\left(1-e^{-\frac{2k-1}{B\sqrt{n}}}\right)^{3}}\frac{1}{B\sqrt{n}}\\
 & =\frac{B^{4}n^{2}}{2}\int_{0}^{\infty}\frac{u^{3}e^{-u}}{\left(1-e^{-u}\right)^{3}}du\left(1+O\left(n^{-\frac{1}{2}}\right)\right)\\
 & =O\left(n^{2}\right),
\end{aligned}
\]
where the constant is independent of $n.$

Now we turn to computing the asymptotic behavior of $f\left(0\right),$
$f^{\prime}\left(0\right)$ and $f^{\prime\prime}\left(0\right),$
which are similar to the proofs of \cite[Proposition 4.4]{BB2} and
\cite[Propoistions 1-3]{RP}. For $f\left(0\right),$ we set $g\left(u\right)\coloneqq-\log\left(1-e^{-\frac{u}{B\sqrt{n}}}\right)$
in Propoistion \ref{p2.2} and rewrite
\[
\begin{aligned}\sum_{k=1}^{m+1}\left(-\log\left(1-e^{-\frac{2k-1}{B\sqrt{n}}}\right)\right) & =\sum_{k=1}^{2\left(m+1\right)}\left(-\log\left(1-e^{-\frac{k}{B\sqrt{n}}}\right)\right)-\sum_{k=1}^{m+1}\left(-\log\left(1-e^{-\frac{2k}{B\sqrt{n}}}\right)\right)\\
 & =\sum_{k=1}^{2\left(m+1\right)}\left(-\log\left(1-e^{-\frac{k}{B\sqrt{n}}}\right)\right)-\sum_{k=1}^{m+1}\left(-\log\left(1-e^{-\frac{k}{\frac{B}{2}\sqrt{n}}}\right)\right).
\end{aligned}
\]
Estimating as in the proofs of \cite[Propoistion 1]{RP} and \cite[(4.12)]{BB2}
we get
\[
\begin{aligned}\sum_{k=1}^{m+1}\left(-\log\left(1-e^{-\frac{2k-1}{B\sqrt{n}}}\right)\right) & =\pi\sqrt{\frac{n}{6}}-\frac{1}{2}e^{-r}-\frac{1}{2}\log\left(B\sqrt{n}\right)-\frac{1}{2}\log\left(2\pi\right)\\
 & -\left(\frac{\pi}{2}\sqrt{\frac{n}{6}}-\frac{1}{4}e^{-r}-\frac{1}{2}\log\left(\frac{B}{2}\sqrt{n}\right)-\frac{1}{2}\log\left(2\pi\right)\right)+o\left(1\right)\\
 & =\frac{\sqrt{n}}{2B}-\frac{1}{4}e^{-r}-\frac{1}{2}\log\left(2\right)+o\left(1\right),
\end{aligned}
\]
where $r=\frac{2m+1-B\sqrt{n}\log\left(2B\sqrt{n}\right)}{B\sqrt{n}}.$
This yields that
\begin{equation}
f\left(0\right)=\pi\sqrt{\frac{2n}{3}}-\frac{1}{2}\log\left(n\right)-r-\frac{1}{2}e^{-r}-\log\left(\frac{4\sqrt{6}}{\pi}\right)+o\left(1\right).\label{eq:4.6}
\end{equation}
Similarly to \cite[(4.13) and (4.14)]{BB2}, we can deduce that
\begin{equation}
f^{\prime}\left(0\right)=O\left(\sqrt{n}\log\left(n\right)\right),\label{eq:4.7}
\end{equation}
\begin{equation}
f^{\prime\prime}\left(0\right)=\frac{2\sqrt{6}}{\pi}n^{\frac{3}{2}}+O\left(n\log\left(n\right)^{2}\right).\label{eq:4.8}
\end{equation}
Hence, the assumption (i) in Proposition \ref{p2.3} holds for any
fixed $\varepsilon>0$ and $\left|\theta\right|\leq\varepsilon n^{\frac{1}{2}}.$

Now, we turn to proving the assumption (ii) in Proposition \ref{p2.3}.
For $\varepsilon n^{\frac{1}{2}}<\left|\theta\right|\leq\frac{1}{2},$
we write
\[
\begin{aligned}\textrm{Re}\left(f\left(2\pi\textrm{i}\theta\right)-f\left(0\right)\right) & =-2\sum_{k=1}^{m+1}\textrm{Re}\left(\textrm{Log}\left(1-e^{-\frac{2k-1}{B\sqrt{n}}+\left(2k-1\right)z}\right)-\textrm{Log}\left(1-e^{-\frac{2k-1}{B\sqrt{n}}}\right)\right)\\
 & =2\sum_{k=1}^{m+1}\sum_{l\geq1}\frac{e^{-l\frac{2k-1}{B\sqrt{n}}}}{l}\left(\cos\left(2\pi\left(2k-1\right)l\theta\right)-1\right)\\
 & \leq2\sum_{k=1}^{m+1}e^{-\frac{2k-1}{B\sqrt{n}}}\left(\cos\left(2\pi\left(2k-1\right)\theta\right)-1\right)
\end{aligned}
\]
From \cite[Propoistion 4.4]{BB2}, we deduce that as $\frac{2m+1}{B\sqrt{n}}\rightarrow\infty,$
\begin{equation}
\begin{aligned} & \sum_{k=1}^{m+1}e^{-\frac{2k-1}{B\sqrt{n}}}\left(\cos\left(2\pi\left(2k-1\right)\theta\right)-1\right)\\
 & =-\frac{B\sqrt{n}}{2}\sum_{k=1}^{m+1}e^{-\frac{2k-1}{B\sqrt{n}}}\left(1-\cos\left(2\pi\left(2k-1\right)\theta\right)\right)\frac{2}{B\sqrt{n}}\\
 & <-\frac{B\sqrt{n}}{4}\int_{0}^{\frac{2m+1}{B\sqrt{n}}}e^{-u}\left(1-\cos\left(2\pi B\sqrt{n}\theta u\right)\right)du\\
 & <-\frac{B\sqrt{n}}{4}\inf_{s\geq\varepsilon}\int_{0}^{T}e^{-u}\left(1-\cos\left(2\pi Bsu\right)\right)du,
\end{aligned}
\label{eq:4.9}
\end{equation}
for any $T>0.$ It is easy to see that the function
\[
s\mapsto\int_{0}^{T}e^{-u}\left(1-\cos\left(2\pi Bsu\right)\right)du
\]
is continuous and nonzero on $\left[\varepsilon,\infty\right).$ Also,
by the Riemann--Lebesgue Lemma, this function tends to $1-e^{-T}>0$
as $s\rightarrow\infty.$ Thus, the right side of $\eqref{eq:4.9}$
is negative, and the assumption (ii) in Proposition \ref{p2.3} holds.
Then, by Proposition \ref{p2.3}, we have
\begin{equation}
ou_{m}\left(n\right)\sim\frac{e^{f\left(0\right)}}{\sqrt{2\pi f^{\prime\prime}\left(0\right)}}.\label{eq:4.10}
\end{equation}
Applying $e^{f\left(0\right)}=\left(q;q^{2}\right)_{m+1}^{-2}q^{\left(2m+1\right)-n},$
we have 
\[
\textrm{\textbf{Q}}_{q,m}\left(N=n\right)\sim\frac{1}{\sqrt{2\pi f^{\prime\prime}\left(0\right)}}.
\]
By \eqref{eq:4.8}, we can prove \eqref{eq:4.3}. By the definition
\[
\textrm{\textbf{P}}_{n}\left(\textrm{PK}=2m+1\right)=\frac{ou_{m}\left(n\right)}{\textrm{ou}\left(n\right)},
\]
\eqref{eq:1.5} and \eqref{eq:4.10}, we can prove \eqref{eq:4.4}.
\end{proof}
In order to handle the numerator appearing in Proposition \,\ref{p4.1},
the next proposition provides an explicit construction of the set
$\ensuremath{B_{n,m}}.$ This construction is valid under a mild condition
on $a_{n}$ and $\ensuremath{b_{n}},$ and will be used directly in
verifying the hypotheses of Proposition \,\ref{p4.1}.
\begin{prop}
\label{p4.3}Suppose that $a_{n}\leq\ensuremath{b_{n}}$ are sequences
of integers such that
\begin{equation}
\sum_{a_{n}\leq k\leq\ensuremath{b_{n}}}\frac{\left(2k-1\right)^{2}q^{2k-1}}{\left(1-q^{2k-1}\right)^{2}}=o\left(c_{n}^{2}\right)\label{eq:4.11}
\end{equation}
holds for a sequence $c_{n}=o\left(n^{\frac{3}{4}}\right).$ For $b_{n,m}\coloneqq\min\left\{ b_{n},m+1\right\} ,$
\[
\begin{aligned}B_{n,m}\coloneqq & \left\{ \left(x_{2k-1}^{\left[j\right]}\right)_{k\in\left[a_{n},b_{n,m}\right],j\in\left\{ L,R\right\} }:\left|\sum_{a_{n}\leq k\leq\ensuremath{b_{n,m}}}\left[\frac{2\left(2k-1\right)q^{2k-1}}{1-q^{2k-1}}-\left(2k-1\right)\left(x_{2k-1}^{\left[L\right]}+x_{2k-1}^{\left[R\right]}\right)\right]\right|\leq c_{n}\right\} \\
 & \times\left\{ 0\right\} ^{2\left(b_{n}-b_{n,m}\right)}
\end{aligned}
\]
satisfies the hypotheses of Proposition \ref{p4.1}, so $d_{\text{TV}}\left(\textrm{\textbf{P}}_{n,m}\left(\boldsymbol{X}_{\left[a_{n},b_{n}\right]}^{-1}\right),\ \textrm{\textbf{Q}}_{q,m}\left(\boldsymbol{X}_{\left[a_{n},b_{n}\right]}^{-1}\right)\right)\to0$
uniformly for $\ensuremath{r\in\left[r_{1},r_{2}\right]}.$
\end{prop}
\begin{proof}
Similarly to the proof of Propoistion 4.6 in \cite{BB2}, we can show
that 
\[
\textrm{\textbf{E}}_{q,m}\left(X_{2k-1}^{\left[L\right]}\right)=\left(1-q^{2k-1}\right)\sum_{l\geq1}lq^{\left(2k-1\right)l}=\frac{q^{2k-1}}{1-q^{2k-1}},
\]
\[
\begin{aligned}\textrm{\textbf{Var}}_{q,m}\left(X_{2k-1}^{\left[L\right]}\right) & =\left(1-q^{2k-1}\right)\sum_{l\geq1}l^{2}q^{\left(2k-1\right)l}-\left(\frac{q^{2k-1}}{1-q^{2k-1}}\right)^{2}\\
 & =\frac{q^{2k-1}}{\left(1-q^{2k-1}\right)^{2}}.
\end{aligned}
\]
These are also exactly the same for $X_{2k-1}^{\left[R\right]}.$
Then we have
\[
\textrm{\textbf{E}}_{q,m}\left(\sum_{a_{n}\leq k\leq\ensuremath{b_{n,m}}}\left(2k-1\right)\left(X_{2k-1}^{\left[L\right]}+X_{2k-1}^{\left[R\right]}\right)\right)=\sum_{a_{n}\leq k\leq\ensuremath{b_{n,m}}}\frac{2\left(2k-1\right)q^{2k-1}}{1-q^{2k-1}},
\]
\[
\textrm{\textbf{Var}}_{q,m}\left(\sum_{a_{n}\leq k\leq\ensuremath{b_{n,m}}}\left(2k-1\right)\left(X_{2k-1}^{\left[L\right]}+X_{2k-1}^{\left[R\right]}\right)\right)=\sum_{a_{n}\leq k\leq\ensuremath{b_{n,m}}}\frac{2\left(2k-1\right)^{2}q^{2k-1}}{\left(1-q^{2k-1}\right)^{2}}.
\]
By \eqref{eq:2.3} and the definition of $B_{n,m},$ we can derive
that
\[
\textrm{\textbf{Q}}_{n,m}\left(\boldsymbol{X}_{\left[a_{n},b_{n}\right]}\in\mathbb{R}^{b_{n}-a_{n}+1}\setminus B_{n,m}\right)\leq c_{n}^{-2}\sum_{a_{n}\leq k\leq\ensuremath{b_{n,m}}}\frac{2\left(2k-1\right)^{2}q^{2k-1}}{\left(1-q^{2k-1}\right)^{2}}=o\left(1\right).
\]
This proves the assumption (i) in Proposition \ref{p4.1}.

Now we set 
\[
\boldsymbol{x}=\left(x_{2k-1}^{\left[j\right]}\right)_{j\in\left\{ L,R\right\} ,a_{n}\leq k\leq\ensuremath{b_{n,m}}}\times\left\{ 0\right\} ^{b_{n}-b_{n,m}}\in B_{n,m},
\]
and
\[
\sum\boldsymbol{x}\coloneqq\sum_{a_{n}\leq k\leq\ensuremath{b_{n,m}}}\left(2k-1\right)\left(x_{2k-1}^{\left[L\right]}+x_{2k-1}^{\left[R\right]}\right).
\]
To prove the assumption (ii) in Proposition \ref{p4.1}, we need to
conclude that
\[
\textrm{\textbf{Q}}_{q,m}\left(n=N\mid\boldsymbol{X}_{\left[a_{n},b_{n}\right]}=\boldsymbol{x}\right)\sim\frac{1}{2^{\frac{5}{4}}\cdot3^{\frac{1}{4}}n^{\frac{3}{4}}}.
\]
It follows from the proof of Propoistion \ref{p4.2} and \cite[Propoistion 4.6]{BB2}
that
\[
\begin{aligned} & \textrm{\textbf{Q}}_{q,m}\left(n=N\mid\boldsymbol{X}_{\left[a_{n},b_{n}\right]}=\boldsymbol{x}\right)\\
 & =\frac{\textrm{\textbf{Q}}_{q,m}\left(n=N\textrm{ and }\boldsymbol{X}_{\left[a_{n},b_{n}\right]}=\boldsymbol{x}\right)}{\textrm{\textbf{Q}}_{q,m}\left(\boldsymbol{X}_{\left[a_{n},b_{n}\right]}=\boldsymbol{x}\right)}\\
 & =\frac{\textrm{\textbf{Q}}_{q,m}\left(n=N\textrm{ and }\boldsymbol{X}_{\left[a_{n},b_{n}\right]}=\boldsymbol{x}\right)}{\prod_{a_{n}\leq k\leq\ensuremath{b_{n,m}}}\textrm{\textbf{Q}}_{q,m}\left(X_{2k-1}^{\left[L\right]}=x_{2k-1}^{\left[L\right]}\right)\textrm{\textbf{Q}}_{q,m}\left(X_{2k-1}^{\left[R\right]}=x_{2k-1}^{\left[R\right]}\right)}\\
 & =\frac{\left(q;q^{2}\right)_{m+1}^{2}q^{n-\left(2m+1\right)}}{q^{\sum\boldsymbol{x}}\frac{\left(q;q^{2}\right)_{\ensuremath{b_{n,m}}+1}^{2}}{\left(q;q^{2}\right)_{a_{n}}^{2}}}\#\left\{ \lambda:N\left(\lambda\right)=n,\textrm{PK}\left(\lambda\right)=2m+1,\boldsymbol{X}_{\left[a_{n},b_{n}\right]}\left(\lambda\right)=\boldsymbol{x}\right\} \\
 & =\frac{\left(q;q^{2}\right)_{m+1}^{2}\left(q;q^{2}\right)_{a_{n}}^{2}q^{n-\left(2m+1\right)}}{q^{\sum\boldsymbol{x}}\left(q;q^{2}\right)_{\ensuremath{b_{n,m}}+1}^{2}}\int_{-\frac{1}{2}}^{\frac{1}{2}}\exp\left(F\left(2\pi\textrm{i}\theta\right)\right)d\theta,
\end{aligned}
\]
where
\[
\begin{aligned}F\left(z\right) & \coloneqq\frac{n-\left(2m+1\right)-\sum\boldsymbol{x}}{B\sqrt{n}}+\left[\sum\boldsymbol{x}+\left(2m+1\right)-n\right]\\
 & -2\sum_{k\in\left[1,m+1\right]\setminus\left[a_{n},\ensuremath{b_{n,m}}\right]}\textrm{Log}\left(1-e^{-\frac{2k-1}{B\sqrt{n}}+\left(2k-1\right)z}\right).
\end{aligned}
\]
Then for $k\in\left[1,m+1\right]\setminus\left[a_{n},\ensuremath{b_{n,m}}\right],$
the computation of all sums proceeds as follows:
\[
F\left(0\right)=\frac{n-\left(2m+1\right)-\sum\boldsymbol{x}}{B\sqrt{n}}-2\sum_{k}\textrm{Log}\left(1-e^{-\frac{2k-1}{B\sqrt{n}}}\right),
\]
\[
F^{\prime}\left(0\right)=\sum\boldsymbol{x}+\left(2m+1\right)-n+2\sum_{k}\frac{2k-1}{e^{\frac{2k-1}{B\sqrt{n}}}-1},
\]
\[
F^{\prime\prime}\left(0\right)=2\sum_{k}\frac{\left(2k-1\right)^{2}e^{-\frac{2k-1}{B\sqrt{n}}}}{\left(1-e^{-\frac{2k-1}{B\sqrt{n}}}\right)^{2}}.
\]
It is easy to see that the exact asymptotic behavior of $F\left(0\right)$
is immaterial, since
\[
\frac{\left(q;q^{2}\right)_{m+1}^{2}\left(q;q^{2}\right)_{a_{n}}^{2}q^{n-\left(2m+1\right)}}{q^{\sum\boldsymbol{x}}\left(q;q^{2}\right)_{\ensuremath{b_{n,m}}+1}^{2}}e^{F\left(0\right)}=1.
\]
For $x\in B_{n,m},$ we have
\[
\begin{aligned}F^{\prime}\left(0\right) & =\left(2m+1\right)-n+2\sum_{k=1}^{m+1}\frac{2k-1}{e^{\frac{2k-1}{B\sqrt{n}}}-1}+o\left(c_{n}\right)\\
 & =f^{\prime}\left(0\right)+o\left(c_{n}\right)\sim f^{\prime}\left(0\right),
\end{aligned}
\]
and
\[
\begin{aligned}F^{\prime\prime}\left(0\right) & =2\sum_{k=1}^{m+1}\frac{\left(2k-1\right)^{2}e^{-\frac{2k-1}{B\sqrt{n}}}}{\left(1-e^{-\frac{2k-1}{B\sqrt{n}}}\right)^{2}}+o\left(c_{n}^{2}\right)\\
 & =f^{\prime\prime}\left(0\right)+o\left(c_{n}^{2}\right)\sim f^{\prime\prime}\left(0\right).
\end{aligned}
\]
By \eqref{eq:4.6}, \eqref{eq:4.7} and \eqref{eq:4.8}, we know that
the assumption (i) in Proposition \ref{p2.3} holds for any fixed
$\varepsilon>0$ and $\left|\theta\right|\leq\varepsilon n^{\frac{1}{2}}.$

Now, we turn to proving the assumption (ii) in Proposition \ref{p2.3}.
We first claim that an interval $\left[\alpha\sqrt{n},\beta\sqrt{n}\right]$
can be found inside $\ensuremath{\left[1,m+1\right]\setminus\left[a_{n},\ensuremath{b_{n,m}}\right]}.$
For large $n$ and $r\in\left[r_{1},r_{2}\right]$ this interval is
indeed contained in $\left[1,m+1\right],$ and the estimate
\[
\sum_{\alpha\sqrt{n}\leq k\leq\beta\sqrt{n}}\frac{\left(2k-1\right)^{2}q^{2k-1}}{\left(1-q^{2k-1}\right)^{2}}\asymp n^{\frac{3}{2}}
\]
follows from the Riemann sum by an integral can be used in the evaluation
of $f^{\prime\prime}\left(0\right).$ It is obvious that
\[
\begin{aligned} & \sum_{k\in\left[1,m+1\right]\setminus\left[a_{n},\ensuremath{b_{n,m}}\right]}e^{-\frac{2k-1}{B\sqrt{n}}}\left(\cos\left(2\pi\left(2k-1\right)\theta\right)-1\right)\\
 & <\sum_{k\in\left[\alpha\sqrt{n},\beta\sqrt{n}\right]}e^{-\frac{2k-1}{B\sqrt{n}}}\left(\cos\left(2\pi\left(2k-1\right)\theta\right)-1\right).
\end{aligned}
\]
Similarly, for $\varepsilon n^{\frac{1}{2}}<\left|\theta\right|\leq\frac{1}{2},$
we can bound this negative sum from above by
\[
-\frac{B\sqrt{n}}{4}\inf_{s\geq\varepsilon}\int_{0}^{T}e^{-u}\left(1-\cos\left(2\pi Bsu\right)\right)du,
\]
which implies that assumption (ii) in Proposition \ref{p2.3} holds.

By Proposition \ref{p2.3}, we can deduce that for $x\in B_{n,m},$
\[
\begin{aligned}\textrm{\textbf{Q}}_{q,m}\left(n=N\mid\boldsymbol{X}_{\left[a_{n},b_{n}\right]}=\boldsymbol{x}\right) & \sim\frac{\left(q;q^{2}\right)_{m+1}^{2}\left(q;q^{2}\right)_{a_{n}}^{2}q^{n-\left(2m+1\right)}}{q^{\sum\boldsymbol{x}}\left(q;q^{2}\right)_{\ensuremath{b_{n,m}}+1}^{2}}e^{F\left(0\right)}\frac{1}{\sqrt{2\pi f^{\prime\prime}\left(0\right)}}\\
 & \sim\frac{1}{\sqrt{2\pi f^{\prime\prime}\left(0\right)}}\sim\textrm{\textbf{Q}}_{q,m}\left(n=N\right).
\end{aligned}
\]
This completes the whole proof.
\end{proof}

\subsection{\label{subsec4.2}Proof of Theorem \ref{t1.3}}

 It follows from \eqref{eq:4.4} that
\begin{equation}
\begin{aligned} & \textrm{\textbf{P}}_{n}\left(\lambda\in\mathcal{OU}\left(n\right):r_{1}\leq\frac{\textrm{PK}\left(\lambda\right)-B\sqrt{n}\log\left(2B\sqrt{n}\right)}{B\sqrt{n}}\leq r_{2}\right)\\
 & \sim\frac{1}{B\sqrt{n}}\sum_{r\in\left[r_{1},r_{2}\right]\cap\left(\frac{1}{B\sqrt{n}}\left(2\mathbb{Z}+1\right)-\log\left(2B\sqrt{n}\right)\right)}e^{-r-\frac{1}{2}e^{-r}}\\
 & =\frac{1}{2}\int_{r_{1}}^{r_{2}}e^{-r-\frac{1}{2}e^{-r}}dr+O\left(n^{-\frac{1}{2}}\right)\\
 & =e^{-\frac{1}{2}e^{-r_{2}}}-e^{-\frac{1}{2}e^{-r_{1}}}+O\left(n^{-\frac{1}{2}}\right).
\end{aligned}
\label{eq:4.12}
\end{equation}
Since $\textrm{\textbf{P}}_{n}$ is a probability measure, we see
that \eqref{eq:4.1} and \eqref{eq:4.2} tend to 0 as $r_{1}\rightarrow-\infty$
and $r_{2}\rightarrow\infty.$ Thus, the first part of Theorem \ref{t1.3}
holds.

To prove the second part, we set
\[
\widetilde{\textrm{PK}}\coloneqq\frac{\textrm{PK}-B\sqrt{n}\log\left(2B\sqrt{n}\right)}{B\sqrt{n}}.
\]
Thus, we only need to prove $\textrm{\textbf{E}}_{n}\left(\widetilde{\textrm{PK}}\right)\sim\gamma-\log\left(2\right).$

We rewrite
\begin{equation}
\textrm{\textbf{E}}_{n}\left(\widetilde{\textrm{PK}}\right)=\left(\sum_{r<r_{1}}+\sum_{r\in\left[r_{1},r_{2}\right]}+\sum_{r>r_{2}}\right)r\textrm{\textbf{P}}_{n}\left(\widetilde{\textrm{PK}}=r\right),\label{eq:4.13}
\end{equation}
where the sums over $r$ are taken over the set $\left(\frac{1}{B\sqrt{n}}\left(2\mathbb{Z}+1\right)-\log\left(2B\sqrt{n}\right)\right).$
By \eqref{eq:4.4}, the second sum in \eqref{eq:4.13} is
\begin{equation}
\begin{aligned} & \frac{1}{B\sqrt{n}}\sum_{r\in\left[r_{1},r_{2}\right]\cap\left(\frac{1}{B\sqrt{n}}\left(2\mathbb{Z}+1\right)-\log\left(2B\sqrt{n}\right)\right)}re^{-r-\frac{1}{2}e^{-r}}\\
 & =\frac{1}{2}\int_{r_{1}}^{r_{2}}re^{-r-\frac{1}{2}e^{-r}}dr+O\left(n^{-\frac{1}{2}}\right).
\end{aligned}
\label{eq:4.14}
\end{equation}
It follows from the proof of Corollary 4.5 in \cite{BB2} that
\begin{equation}
\begin{aligned} & \sum_{r>r_{2}}r\textrm{\textbf{P}}_{n}\left(\widetilde{\textrm{PK}}=r\right)\\
 & =\frac{1}{B\sqrt{n}}\sum_{r>r_{2}}\textrm{\textbf{P}}_{n}\left(\textrm{\ensuremath{\widetilde{\textrm{PK}}}}\geq r_{2}\right)+r_{2}\textrm{\textbf{P}}_{n}\left(\textrm{\ensuremath{\widetilde{\textrm{PK}}}}>r_{2}\right)\\
 & =\int_{r_{2}}^{\infty}\left(1-e^{-\frac{1}{2}e^{-r}}\right)dr+r_{2}\left(1-e^{-\frac{1}{2}e^{-r_{2}}}\right)+O\left(n^{-\frac{1}{2}}\right),
\end{aligned}
\label{eq:4.15}
\end{equation}
and
\begin{equation}
\begin{aligned} & \sum_{r<r_{1}}r\textrm{\textbf{P}}_{n}\left(\widetilde{\textrm{PK}}=r\right)\\
 & =\frac{1}{B\sqrt{n}}\sum_{r<r_{1}}\textrm{\textbf{P}}_{n}\left(\textrm{\ensuremath{\widetilde{\textrm{PK}}}}\leq r_{1}\right)+r_{1}\textrm{\textbf{P}}_{n}\left(\textrm{\ensuremath{\widetilde{\textrm{PK}}}}>r_{1}\right)\\
 & =\int_{-\infty}^{r_{1}}e^{-\frac{1}{2}e^{-r}}dr+r_{1}e^{-\frac{1}{2}e^{-r_{1}}}+O\left(n^{-\frac{1}{2}}\right).
\end{aligned}
\label{eq:4.16}
\end{equation}
By \eqref{eq:4.14}, \eqref{eq:4.15} and \eqref{eq:4.16}, we can
deduce that as $r_{1}\rightarrow-\infty$ and $r_{2}\rightarrow\infty,$
\[
\begin{aligned}\textrm{\textbf{E}}_{n}\left(\widetilde{\textrm{PK}}\right) & \sim\frac{1}{2}\int_{-\infty}^{\infty}re^{-r-\frac{1}{2}e^{-r}}dr+O\left(n^{-\frac{1}{2}}\right)\\
 & =\frac{1}{2}\int_{\infty}^{0}\left(-\log\left(2t\right)\right)e^{-\left(-\log\left(2t\right)\right)-t}\left(-\frac{1}{t}\right)dt\\
 & =-\int_{0}^{\infty}\log\left(2t\right)e^{-t}dt\\
 & =-\int_{0}^{\infty}\log\left(t\right)e^{-t}dt-\log\left(2\right)\int_{0}^{\infty}e^{-t}dt\\
 & =\gamma-\log\left(2\right).
\end{aligned}
\]
This finishes the proof. \qed

\subsection{\label{subsec4.3}Proof of Theorem \ref{t1.4}}

Set $t_{n}=o\left(n^{\frac{1}{4}}\right)$ and 
\[
\boldsymbol{W}_{n}\coloneqq\left(\frac{Y_{2k-1}^{\left[j\right]}-B\sqrt{n}\log\left(2B\sqrt{n}\right)}{B\sqrt{n}}\right)_{j\in\left\{ L,R\right\} ,1\leq t\leq t_{n}}.
\]
We also define $\xi_{n,m}\coloneqq\textrm{\textbf{P}}_{n,m}\left(\boldsymbol{W}_{n}^{-1}\right),$
$\zeta_{n,m}\coloneqq\textrm{\textbf{Q}}_{q,m}\left(\boldsymbol{W}_{n}^{-1}\right),$
and let $\boldsymbol{\nu}_{n,r}$ be the probability measure on $\left(-\infty,r\right]^{2t_{n}}$
given by the density
\[
\begin{cases}
\frac{1}{4^{2t_{n}}}e^{-\frac{1}{2}e^{-r}-\sum_{t=1}^{t_{n}}\left(u_{2t-1}^{\left[L\right]}+u_{2t-1}^{\left[R\right]}\right)-\frac{1}{4}e^{-u_{2t_{n}-1}^{\left[L\right]}}-\frac{1}{4}e^{-u_{2t_{n}-1}^{\left[R\right]}}} & \textrm{if }u_{1}^{\left[j\right]}\geq\ldots\geq u_{2t_{n}-1}^{\left[j\right]},\\
0 & \textrm{otherwise,}
\end{cases}
\]
for $j\in\left\{ L,R\right\} .$

We first prove that $\zeta_{n,m}\left(U\right)\sim\boldsymbol{\nu}_{n,r}\left(U\right)$
uniformly for $r\in\left[r_{1},r_{2}\right]$ and
\begin{equation}
U=\underset{\substack{1\leq t\leq t_{n}\\
j\in\left\{ L,R\right\} 
}
}{\mathop{\prod}}\left(-\infty,v_{2t-1}^{\left[j\right]}\right].\label{eq:4.17}
\end{equation}
Define $\boldsymbol{\omega}\coloneqq\left(\omega_{2t-1}^{\left[j\right]}\right)_{j\in\left\{ L,R\right\} ,1\leq t\leq t_{n}}$.
Then, for $j\in\left\{ L,R\right\} ,$
\[
r\geq\omega_{1}^{\left[j\right]}\geq\cdots\geq\omega_{2t-1}^{\left[j\right]},
\]
and
\[
\begin{aligned}y_{2t-1}^{\left[j\right]} & =2y^{\left[j\right]}\left(t\right)+1\\
 & \coloneqq B\sqrt{n}\left(\omega_{2t-1}^{\left[j\right]}+\log\left(2B\sqrt{n}\right)\right)\\
 & \in2\mathbb{Z}+1.
\end{aligned}
\]
Hence, we have
\[
\begin{aligned}\zeta_{n,m}\left(\boldsymbol{\omega}\right) & =\textrm{\textbf{Q}}_{q,m}\left(\left(Y_{2k-1}^{\left[L\right]}\right)_{1\leq t\leq t_{n}}\times\left(Y_{2k-1}^{\left[R\right]}\right)_{1\leq t\leq t_{n}}=\left(y_{2k-1}^{\left[L\right]}\right)_{1\leq t\leq t_{n}}\times\left(y_{2k-1}^{\left[R\right]}\right)_{1\leq t\leq t_{n}}\right)\\
 & =q^{-\left(2m+1\right)}\left(q;q^{2}\right)_{m+1}^{2}q^{\left(2m+1\right)+\sum_{t=1}^{t_{n}}\left(y_{2t-1}^{\left[L\right]}+y_{2t-1}^{\left[R\right]}\right)}\sum_{\lambda}q^{\left|\lambda\right|},
\end{aligned}
\]
where the sum is taken over all pairs of partitions $\lambda^{\left[L\right]}$
and $\lambda^{\left[R\right]}$ such that the parts of $\lambda^{\left[j\right]}$
do not exceed $y_{2t_{n}-1}^{\left[j\right]}$ for $\ensuremath{j\in\left\{ L,R\right\} }.$
Using Lemma\,\ref{l2.8}, we obtain that
\[
\begin{aligned} & \frac{q^{\sum_{t=1}^{t_{n}}\left(y_{2t-1}^{\left[L\right]}+y_{2t-1}^{\left[R\right]}\right)}\left(q;q^{2}\right)_{m+1}^{2}}{\left(q;q^{2}\right)_{y^{\left[L\right]}\left(t_{n}\right)+1}^{2}\left(q;q^{2}\right)_{y^{\left[R\right]}\left(t_{n}\right)+1}^{2}}\\
 & =q^{\sum_{t=1}^{t_{n}}\left(y_{2t-1}^{\left[L\right]}+y_{2t-1}^{\left[R\right]}\right)}\prod_{y^{\left[L\right]}\left(t_{n}\right)+1<t\leq m+1}\left(1-q^{2t-1}\right)\prod_{y^{\left[R\right]}\left(t_{n}\right)+1<t\leq m+1}\left(1-q^{2t-1}\right)\\
 & =e^{-\sum_{t=1}^{t_{n}}\left(\omega_{2t-1}^{\left[L\right]}+\omega_{2t-1}^{\left[R\right]}\right)}\left(\frac{1}{2B\sqrt{n}}\right)^{2t_{n}}\prod_{y^{\left[L\right]}\left(t_{n}\right)+1<t\leq m+1}\left(1-q^{2t-1}\right)\\
 & \times\prod_{y^{\left[R\right]}\left(t_{n}\right)+1<t\leq m+1}\left(1-q^{2t-1}\right)\\
 & \sim e^{-\sum_{t=1}^{t_{n}}\left(\omega_{2t-1}^{\left[L\right]}+\omega_{2t-1}^{\left[R\right]}\right)}\left(\frac{1}{2B\sqrt{n}}\right)^{2t_{n}}e^{-\frac{1}{2}e^{-r}-\frac{1}{4}e^{-\omega_{2t_{n}-1}^{\left[L\right]}}-\frac{1}{4}e^{-\omega_{2t_{n}-1}^{\left[R\right]}}},
\end{aligned}
\]
for $r,\omega_{2t-1}^{\left[L\right]},\omega_{2t-1}^{\left[R\right]}\geq-\frac{\log\left(n\right)}{8}.$
Set
\[
S\coloneqq\left\{ \boldsymbol{\omega}:\omega_{2t-1}^{\left[j\right]}\geq-\frac{\log\left(n\right)}{8}\textrm{ for }j\in\left\{ L,R\right\} \right\} .
\]
Note that $\omega_{2t-1}^{\left[j\right]}\in\frac{1}{B\sqrt{n}}\left[\left(2\mathbb{Z}+1\right)-\log\left(2B\sqrt{n}\right)\right].$
Then we can derive that
\[
\zeta_{n,m}\left(U\cap S\right)\sim\boldsymbol{\nu}_{n,r}\left(U\cap S\right)
\]
for $r\in\left[r_{1},r_{2}\right]$ by the Riemann sum and \eqref{eq:4.17}.
Similarly to \cite{BB2,FB}, we also have $\zeta_{n,m}\left(S\right)\sim\boldsymbol{\nu}_{n,r}\left(S\right),$
$\boldsymbol{\nu}_{n,r}\left(S^{c}\right)\rightarrow0$ and $\zeta_{n,m}\left(S^{c}\right)\rightarrow0.$
Therefore, we can deduce that $\zeta_{n,m}\left(U\right)\sim\boldsymbol{\nu}_{n,r}\left(U\right)$
uniformly for $r\in\left[r_{1},r_{2}\right].$

By definition, we have
\[
Y_{t}^{\left[j\right]}=\sup\left\{ k\ge1:\sum_{i\ge k}X_{i}^{\left[j\right]}\ge t\right\} .
\]
Consequently, $\boldsymbol{W}_{n}\in S$ holds precisely when $Y_{2t_{n}-1}^{\left[j\right]}\ge B\sqrt{n}\log\left(2Bn^{3/8}\right)$
for $j\in\left\{ L,R\right\} .$ Equivalently, $Y_{t_{n}}^{\left[j\right]}$
is determined solely by those $X_{k}^{\left[j\right]}$ with $\ensuremath{k\ge B\sqrt{n}\log\left(2Bn^{3/8}\right)}.$
Let $a_{n}=B\sqrt{n}\log\left(2Bn^{3/8}\right)$ and $b_{n}=n.$ Then
\[
\sum_{a_{n}\leq k\leq\ensuremath{b_{n}}}\frac{\left(2k-1\right)^{2}q^{2k-1}}{\left(1-q^{2k-1}\right)^{2}}\leq\frac{B^{3}n^{\frac{3}{2}}}{8}\int_{a_{n}}^{\infty}\frac{u^{2}e^{-u}}{\left(1-e^{-u}\right)}du=o\left(n^{\frac{3}{2}}\right).
\]
Thus, by Propoistion \ref{p4.3}, we have $\xi_{n,m}\left(U\cap S\right)\sim\zeta_{n,m}\left(U\cap S\right).$
From this, together with the steps leading to the analogous result
for the full set $U,$ we deduce 
\[
\xi_{n,m}\left(U\right)\sim\zeta_{n,m}\left(U\right)\sim\boldsymbol{\nu}_{n,r}\left(U\right).
\]
By \eqref{eq:4.13}, we can derive that
\begin{equation}
\begin{aligned} & \left|\sum_{r}\xi_{n,m}\left(U\right)\textrm{\textbf{P}}_{n}\left(\textrm{PK}=2m+1\right)-\boldsymbol{\nu}_{n,r}\left(U\right)\right|\\
 & =\left|\sum_{r}\left(\xi_{n,m}\left(U\right)-\boldsymbol{\nu}_{n,r}\left(U\right)\right)\textrm{\textbf{P}}_{n}\left(\textrm{PK}=2m+1\right)\right|\\
 & \leq\left(\sum_{r<r_{1}}+\sum_{r\in\left[r_{1},r_{2}\right]}+\sum_{r>r_{2}}\right)\left|\xi_{n,m}\left(U\right)-\boldsymbol{\nu}_{n,r}\left(U\right)\right|\textrm{\textbf{P}}_{n}\left(\textrm{PK}=2m+1\right)\\
 & \leq2e^{-\frac{1}{2}e^{-r_{1}}}+\sum_{r\in\left[r_{1},r_{2}\right]}\left|\xi_{n,m}\left(U\right)-\boldsymbol{\nu}_{n,r}\left(U\right)\right|\textrm{\textbf{P}}_{n}\left(\textrm{PK}=2m+1\right)+2\left(1-e^{-\frac{1}{2}e^{-r_{2}}}\right)\\
 & \sim2e^{-\frac{1}{2}e^{-r_{1}}}+2\left(1-e^{-\frac{1}{2}e^{-r_{2}}}\right).
\end{aligned}
\label{eq:4.18}
\end{equation}
Taking $r_{1}\rightarrow-\infty$ and $r_{2}\rightarrow\infty,$ we
finally complete the proof of Theorem \ref{t1.4}. \qed

\subsection{\label{subsec4.4}Proof of Theorem \ref{t1.5}}

Let $a_{n}=1$ and $b_{n}=k_{n}=o\left(n^{\frac{1}{4}}\right).$
Then we have
\[
\sum_{1\leq k\leq k_{n}}\frac{\left(2k-1\right)^{2}q^{2k-1}}{\left(1-q^{2k-1}\right)^{2}}\sim\frac{B^{3}n^{\frac{3}{2}}}{8}\int_{0}^{\frac{2k_{n}-1}{B\sqrt{n}}}\frac{u^{2}e^{-u}}{\left(1-e^{-u}\right)}du=o\left(n^{\frac{3}{2}}\right).
\]
Set $\boldsymbol{x}\coloneqq\left(x_{2k-1}^{\left[j\right]}\right)_{j\in\left\{ L,R\right\} ,1\leq k\leq k_{n}}\in\mathbb{R}_{\geq0}^{2k_{n}}.$
Then 
\[
x_{2k-1}^{\left[j\right]}\in\frac{2k-1}{B\sqrt{n}}\mathbb{N}.
\]
 Similarly to \cite[(5.1)]{BB2}, we get
\begin{equation}
\begin{aligned}\textrm{\textbf{Q}}_{q,m}\left(\boldsymbol{X}_{\left[1,k_{n}\right]}=\boldsymbol{x}\right) & =\underset{\substack{1\leq k\leq k_{n}\\
j\in\left\{ L,R\right\} 
}
}{\mathop{\prod}}\textrm{\textbf{Q}}_{q,m}\left(X_{2k-1}^{\left[j\right]}=\frac{B\sqrt{n}}{2k-1}x_{2k-1}^{\left[j\right]}\right)\\
 & =\underset{\substack{1\leq k\leq k_{n}\\
j\in\left\{ L,R\right\} 
}
}{\mathop{\prod}}\left(1-q^{2k-1}\right)q^{\frac{B\sqrt{n}}{2k-1}\left(2k-1\right)\cdot x_{2k-1}^{\left[j\right]}}\\
 & \sim\left[\left(2k_{n}-1\right)!!\right]^{2}\left(\frac{1}{B\sqrt{n}}\right)^{2k_{n}}\underset{\substack{1\leq k\leq k_{n}\\
j\in\left\{ L,R\right\} 
}
}{\mathop{\prod}}e^{-x_{2k-1}^{\left[j\right]}}\\
 & =\underset{\substack{1\leq k\leq k_{n}\\
j\in\left\{ L,R\right\} 
}
}{\mathop{\prod}}e^{-x_{2k-1}^{\left[j\right]}}\frac{2k-1}{B\sqrt{n}}.
\end{aligned}
\label{eq:4.19}
\end{equation}
Then, for $r\in\left[r_{1},r_{2}\right]$ and any set
\[
B=\underset{\substack{1\leq k\leq k_{n}\\
j\in\left\{ L,R\right\} 
}
}{\mathop{\prod}}\left(-\infty,v_{2k-1}^{\left[j\right]}\right]\subset\mathbb{R}^{2k_{n}},
\]
we apply Propoistion \ref{p4.3} to obtain that
\[
\begin{aligned}\textrm{\textbf{P}}_{n,m}\left(\boldsymbol{X}_{\left[1,k_{n}\right]}\in B\right) & \sim\textrm{\textbf{Q}}_{q,m}\left(\boldsymbol{X}_{\left[1,k_{n}\right]}\in B\right)\\
 & \sim\sum_{\omega\in B}\underset{\substack{1\leq k\leq k_{n}\\
j\in\left\{ L,R\right\} 
}
}{\mathop{\prod}}e^{-\omega_{2k-1}^{\left[j\right]}}\frac{2k-1}{B\sqrt{n}}\\
 & \sim\underset{\substack{1\leq k\leq k_{n}\\
j\in\left\{ L,R\right\} 
}
}{\mathop{\prod}}\int_{-\infty}^{v_{2k-1}^{\left[j\right]}}\left(\frac{1}{2}e^{-u_{2k-1}^{\left[j\right]}}\right)du_{2k-1}^{\left[j\right]}\\
 & \eqqcolon\boldsymbol{\nu}_{n}\left(B\right).
\end{aligned}
\]
It follows from \eqref{eq:4.18} that as $r_{1}\rightarrow-\infty$
and $r_{2}\rightarrow\infty,$
\[
\begin{aligned} & \left|\textrm{\textbf{P}}_{n}\left(\boldsymbol{X}_{\left[1,k_{n}\right]}\in B\right)-\boldsymbol{\nu}_{n}\left(B\right)\right|\\
 & =\left|\sum_{r}\textrm{\textbf{P}}_{n,m}\left(\boldsymbol{X}_{\left[1,k_{n}\right]}\in B\right)\textrm{\textbf{P}}_{n}\left(\textrm{PK}=2m+1\right)-\boldsymbol{\nu}_{n}\left(B\right)\right|\\
 & \rightarrow0.
\end{aligned}
\]
This shows the first part of Theorem \ref{t1.5}.

Similarly, we have
\[
\begin{aligned}\textrm{\textbf{Q}}_{q,m}\left(\frac{\left(2k-1\right)X_{2k-1}^{\left[L\right]}}{B\sqrt{n}}=\omega^{\left[L\right]},\frac{\left(2k-1\right)X_{2k-1}^{\left[R\right]}}{B\sqrt{n}}=\omega^{\left[R\right]}\right) & =\left(1-e^{-\frac{2k-1}{B\sqrt{n}}}\right)^{2}e^{-\omega^{\left[L\right]}-\omega^{\left[R\right]}}\\
 & \sim\left(\frac{2k-1}{B\sqrt{n}}\right)^{2}e^{-\omega^{\left[L\right]}-\omega^{\left[R\right]}}
\end{aligned}
\]
for $k=o\left(n^{\frac{1}{2}}\right)$ and $\omega^{\left[L\right]},\omega^{\left[R\right]}\in\frac{2k-1}{B\sqrt{n}}\mathbb{N}_{0}.$
This proves the second part.

Also
\[
\textrm{\textbf{Q}}_{q,m}\left(\left(2k-1\right)X_{2k-1}^{\left[L\right]}=u^{\left[L\right]},\left(2k-1\right)X_{2k-1}^{\left[R\right]}=u^{\left[R\right]}\right)\sim\left(1-e^{-\frac{c}{B}}\right)^{2}e^{-\frac{c}{B}\left(u^{\left[L\right]}+u^{\left[R\right]}\right)}
\]
for $2k-1=\left\lfloor c\sqrt{n}\right\rfloor $ and $u^{\left[L\right]},u^{\left[R\right]}\in\mathbb{N}_{0}.$
This proves the last part. \qed

\subsection{\label{subsec4.5}Proof of Theorem \ref{t1.6}}

Similarly to Subsection \ref{subsec4.4}, apply Propoistion \ref{p4.3}
to $\boldsymbol{X}_{\left[1,k_{n}\right]}$ with $k_{n}=o\left(n^{\frac{1}{2}}\right).$
It is easy to see that 
\[
d_{\text{TV}}\left(\textrm{\textbf{P}}_{n,m}\left(W_{n}^{-1}\right),\ \textrm{\textbf{Q}}_{q,m}\left(W_{n}^{-1}\right)\right)\to0,
\]
where
\[
W_{n}\coloneqq\left(\frac{\sum_{1\leq k\leq k_{n}}X_{2k-1}^{\left[j\right]}-B\sqrt{n}k\log\left(2k_{n}-1\right)}{B\sqrt{n}}\right)_{j\in\left\{ L,R\right\} }.
\]
We rewrite $\textrm{\textbf{P}}_{n}$ as 
\[
\begin{aligned} & \textrm{\textbf{P}}_{n}\left[\left(W_{n}\right)_{j}\leq v_{j},j\in\left\{ L,R\right\} \right]\\
 & =\left(\sum_{r<r_{1}}+\sum_{r\in\left[r_{1},r_{2}\right]}+\sum_{r>r_{2}}\right)\textrm{\textbf{P}}_{n,m}\left[\left(W_{n}\right)_{j}\leq v_{j},j\in\left\{ L,R\right\} \right]\textrm{\textbf{P}}_{n}\left(\textrm{PK}=2m+1\right).
\end{aligned}
\]
Similarly to \eqref{eq:4.15} and \eqref{eq:4.16}, the sums over
$r<r_{1}$ and $r>r_{2}$ tend to 0. In the second sum, we can replace
$\textrm{\textbf{P}}_{n,m}$ with $\textrm{\textbf{Q}}_{q,m}.$ By
the independence of $X_{2k-1}^{\left[j\right]}$ in the Conditioned
Boltzmann model, we have
\begin{equation}
\sum_{r\in\left[r_{1},r_{2}\right]}\textrm{\textbf{P}}_{n,m}\left(\textrm{PK}=2m+1\right)\underset{\substack{j\in\left\{ L,R\right\} }
}{\mathop{\prod}}\textrm{\textbf{Q}}_{q,m}\left[\left(W_{n}\right)_{j}\leq v_{j}\right].\label{eq:4.20}
\end{equation}

In accordance with Section 8 of \cite{FB}, we proceed by examining
a particular term of $\textrm{\textbf{Q}}_{q,m}$ and begin by restricting
the index range to $\ensuremath{k\le\kappa_{n}},$ with $\kappa_{n}\coloneqq\left\lfloor k_{n}^{\frac{1}{2}}\right\rfloor =o\left(n^{\frac{1}{4}}\right),$
a condition that permits the use of the approximation \eqref{eq:4.19}.
For convenience, we first perform this calculation omitting the $B\sqrt{n}\log\left(2k_{n}-1\right)$
correction term. It yields that
\[
\begin{aligned} & \textrm{\textbf{Q}}_{q,m}\left(\frac{\sum_{1\leq k\leq\kappa_{n}}X_{2k-1}^{\left[j\right]}}{B\sqrt{n}}\leq v_{j}\right)\\
 & =\sum_{\omega_{2\kappa_{n}-1}^{\left[j\right]}\in\frac{1}{B\sqrt{n}}\mathbb{N}_{0}\cap\left[0,v_{j}\right]}\cdots\sum_{\omega_{1}^{\left[j\right]}\in\frac{1}{B\sqrt{n}}\mathbb{N}_{0}\cap\left[0,v_{j}-\omega_{2\kappa_{n}-1}^{\left[j\right]}-\cdots-\omega_{3}^{\left[j\right]}\right]}\prod_{1\leq k\leq\kappa_{n}}\textrm{\textbf{Q}}_{q,m}\left(\frac{X_{2k-1}^{\left[j\right]}}{B\sqrt{n}}=\omega_{2k-1}^{\left[j\right]}\right)
\end{aligned}
\]
By \eqref{eq:4.19}, we have
\[
\begin{aligned}\prod_{1\leq k\leq\kappa_{n}}\textrm{\textbf{Q}}_{q,m}\left(\frac{X_{2k-1}^{\left[j\right]}}{B\sqrt{n}}=\omega_{2k-1}^{\left[j\right]}\right) & \sim\prod_{1\leq k\leq\kappa_{n}}\textrm{\textbf{Q}}_{q,m}\left(X_{2k-1}^{\left[j\right]}=\frac{B\sqrt{n}}{2k-1}\left(2k-1\right)\cdot\omega_{2k-1}^{\left[j\right]}\right)\\
 & \sim\left(2\kappa_{n}-1\right)!!\prod_{1\leq k\leq\kappa_{n}}\left(e^{-\left(2k-1\right)\omega_{2k-1}^{\left[j\right]}}\frac{1}{B\sqrt{n}}\right).
\end{aligned}
\]
Hence, we can deduce that
\[
\begin{aligned} & \textrm{\textbf{Q}}_{q,m}\left(\frac{\sum_{1\leq k\leq\kappa_{n}}X_{2k-1}^{\left[j\right]}}{B\sqrt{n}}\leq v_{j}\right)\\
 & \sim\left(2\kappa_{n}-1\right)!!\int_{0}^{v_{j}}\cdots\int_{0}^{v_{j}-\omega_{2\kappa_{n}-1}^{\left[j\right]}-\cdots-\omega_{3}^{\left[j\right]}}e^{-u_{1}^{\left[j\right]}-\cdots-\left(2\kappa_{n}-1\right)u_{2\kappa_{n}-1}^{\left[j\right]}}du_{1}^{\left[j\right]}\cdots du_{2\kappa_{n}-1}^{\left[j\right]}\\
 & =\left(1-e^{-v_{j}}\right)^{\kappa_{n}},
\end{aligned}
\]
which is uniform in $v_{j}\in\left[0,\infty\right).$ We replace $v_{j}$
with $v_{j}+\log\left(2\kappa_{n}-1\right)$ for fixed $v_{j}$ to
get
\[
\textrm{\textbf{Q}}_{q,m}\left(\frac{\sum_{1\leq k\leq\kappa_{n}}X_{2k-1}^{\left[j\right]}-B\sqrt{n}\log\left(2\kappa_{n}-1\right)}{B\sqrt{n}}\leq v_{j}\right)\sim e^{-\frac{1}{2}e^{-v_{j}}}.
\]
For $k_{n}^{\frac{1}{2}}\leq k\leq k_{n},$ the same reasoning that
leads to Theorem 1.8 in \cite{BB2} gives
\[
\begin{aligned} & \textrm{\textbf{Var}}_{q,m}\left(\frac{\sum_{k_{n}^{\frac{1}{2}}\leq k\leq k_{n}}X_{2k-1}^{\left[j\right]}-B\sqrt{n}\log\left(2k_{n}^{\frac{1}{2}}-1\right)}{B\sqrt{n}}\right)\\
 & =\frac{1}{B^{2}n}\sum_{k_{n}^{\frac{1}{2}}\leq k\leq k_{n}}\frac{q^{2k-1}}{\left(1-q^{2k-1}\right)^{2}}\sim\frac{1}{B^{2}n}\sum_{k_{n}^{\frac{1}{2}}\leq k\leq k_{n}}\frac{1}{\left(1-e^{-\frac{2k-1}{B\sqrt{n}}}\right)^{2}}\\
 & \sim\frac{1}{B^{2}n}\sum_{k_{n}^{\frac{1}{2}}\leq k\leq k_{n}}\frac{1}{\left(2k-1\right)^{2}}=o\left(1\right),
\end{aligned}
\]
and
\[
\begin{aligned} & \textrm{\textbf{E}}_{q,m}\left(\frac{\sum_{k_{n}^{\frac{1}{2}}\leq k\leq k_{n}}X_{2k-1}^{\left[j\right]}-B\sqrt{n}\log\left(2k_{n}^{\frac{1}{2}}-1\right)}{B\sqrt{n}}\right)\\
 & =\frac{1}{B\sqrt{n}}\sum_{k_{n}^{\frac{1}{2}}\leq k\leq k_{n}}\frac{q^{2k-1}}{1-q^{2k-1}}\sim\frac{1}{B\sqrt{n}}\sum_{k_{n}^{\frac{1}{2}}\leq k\leq k_{n}}\frac{1}{2k-1}+o\left(1\right)\\
 & =\log\left(2k_{n}^{\frac{1}{2}}-1\right)+o\left(1\right).
\end{aligned}
\]
Therefore, we can derive that
\[
\begin{aligned}\textrm{\textbf{Q}}_{q,m}\left[\left(W_{n}\right)_{j}\leq v_{j}\right] & =\textrm{\textbf{Q}}_{q,m}\left(\frac{\sum_{1\leq k\leq k_{n}}X_{2k-1}^{\left[j\right]}-B\sqrt{n}\log\left(2k_{n}-1\right)}{B\sqrt{n}}\leq v_{j}\right)\\
 & \sim e^{-\frac{1}{2}e^{-v_{j}}}.
\end{aligned}
\]
By Propoistion \ref{p4.2}, we can change \eqref{eq:4.20} into
\[
\sum_{r\in\left[r_{1},r_{2}\right]}\textrm{\textbf{P}}_{n,m}\left(\textrm{PK}=2m+1\right)e^{-\frac{1}{2}\left(e^{-v_{L}}+e^{-v_{R}}\right)}\sim\left(e^{-\frac{1}{2}e^{-r_{2}}}-e^{-\frac{1}{2}e^{-r_{1}}}\right)e^{-\frac{1}{2}\left(e^{-v_{L}}+e^{-v_{R}}\right)}.
\]
Taking $r_{1}\rightarrow-\infty$ and $r_{2}\rightarrow\infty$ completes
the proof. \qed 

\section*{Acknowledgement}

This work was partially supported by the Scientific Research Fund
of Hunan Provincial Education Department (Grant No. 25B0010).

\end{document}